\documentclass{amsart}
\usepackage{amsmath,amsfonts,amssymb,amscd,amsthm,epsf}
\usepackage{pinlabel}
\newtheorem{prop}{Proposition}[section]
\newtheorem{thm}[prop]{Theorem}
\newtheorem{cor}[prop]{Corollary}

\newtheorem{lem}[prop]{Lemma}

\newtheorem{adden}[prop]{Addendum}

\theoremstyle{definition}
\newtheorem{de}[prop]{Definition}

\theoremstyle{remark}
\newtheorem{Remark}[prop]{Remark}             
\newtheorem{Remarks}[prop]{Remarks}             
\def\complex{{\mathbb C}}
\def\C{{\mathbb C}}
\def\CP{{\mathbb C \mathbb P}}
\def\zed{{\mathbb Z}}
\def\Z{{\mathbb Z}}
\def\real{{\mathbb R}}
\def\R{{\mathbb R}}

\def\inter{\mathop{\rm int}\nolimits}
\def\cl{\mathop{\rm cl}\nolimits}

\def\dim{\mathop{\rm dim}\nolimits}
\def\rel{\mathop{\rm rel}\nolimits}

\def\max{\mathop{\rm max}\nolimits}
\def\id{\mathop{\rm id}\nolimits}

\def\SO{\mathop{\rm SO}\nolimits}

\def\G{\mathop{\mathcal{G}}\nolimits}
\def\r{\mathop{\mathcal{R}}\nolimits}
\def\co{\colon\thinspace}

\begin{document}
\title{Creating Stein surfaces by topological isotopy}
\author{Robert E. Gompf}
\thanks{Partially supported by NSF grants DMS-0102922 and DMS-0603958.}
\address{
The University of Texas at Austin
}
\email{gompf@math.utexas.edu}
\begin{abstract}
We combine Freedman's topology with Eliashberg's holomorphic theory to construct Stein neighborhood systems in complex surfaces, and use these to study various notions of convexity and concavity. Every tame, topologically embedded 2-complex $K$ in a complex surface, after $C^0$-small topological ambient isotopy, is the intersection of an uncountable nested family of Stein regular neighborhoods that are all topologically ambiently isotopic rel $K$, but frequently realize uncountably many diffeomorphism types. These arise from the Cantor set levels of a topological mapping cylinder. The boundaries of the neighborhoods are 3-manifolds that are only topologically embedded, but still satisfy a notion of pseudoconvexity. Such 3-manifolds share some basic properties of hypersurfaces that are strictly pseudoconvex in the usual smooth sense, but they are far more common. The complementary notion of topological pseudoconcavity is realized by uncountably many diffeomorphism types homeomorphic to $\R^4$.
\end{abstract}
\maketitle


\section{Introduction}

Since its inception, 4-manifold theory has been repeatedly transformed by breakthroughs from disparate areas of mathematics. Freedman's original breakthrough \cite{F} (see also \cite{FQ}) around 1981 used delicate infinite surface constructions and difficult point-set topology to elucidate the topological category, classifying simply connected, closed, topological 4-manifolds up to homeomorphism. Contemporaneously, Donaldson \cite{D} pioneered the use of gauge theory to understand smooth 4-manifolds, showing that they were very different from their topological counterparts. The interplay between these theories showed, for example, that smoothable 4-manifolds typically admit infinitely many nondiffeomorphic smooth structures, even when higher dimensional smoothing theory predicts uniqueness. For example, $\R^4$ admits uncountably many exotic smoothings, while $\R^n$ is uniquely smoothable for $n\ne 4$. A third development was Eliashberg's 1990 paper \cite{E} (see also \cite{CE}) on Stein manifolds, that is, complex $n$-manifolds admitting proper holomorphic embeddings in $\C^N$. Eliashberg showed, surprisingly, that whether a smooth manifold admits a Stein structure is entirely determined by basic differential topology. While the characterization is particularly subtle in real dimension 4 (Stein {\em surfaces}), its interplay with gauge-theoretic results has had striking applications both in complex analysis and smooth 4-manifold topology. What is perhaps more startling is that Eliashberg's work can be played off against Freedman's topological theory to obtain results spanning the topological, smooth and complex analytic categories. For example, every Stein open subset of $\C^2$ ({\em domain of holomorphy}) contains an uncountable family of other such  domains of holomorphy, all of which are homeomorphic to the original but pairwise nondiffeomorphic \cite{MinGen} (also Corollary~\ref{C2} and subsequent text). In particular, there are uncountably many diffeomorphism types of domains of holomorphy in $\C^2$ homeomorphic to $\R^4$. The present paper uses the full depth of Freedman's construction to further explore the interplay of these three categories, for example, constructing uncountable neighborhood systems of homeomorphic but nondiffeomorphic Stein surfaces.

The main results of this paper were announced in the 2005 article \cite{JSG}, and then again in the expository paper \cite{yfest}. The first of these papers showed that in a complex surface $X$, every tame, topologically embedded CW 2-complex can be perturbed to have a topological regular neighborhood that is Stein in the complex structure inherited as an open subset of $X$. (We augment this to a Stein neighborhood system in Corollary~\ref{2cplx} below.) Its antecedent \cite{Ann} used Eliashberg's work to construct Stein surfaces in the abstract (rather than as subsets of complex surfaces), and mainly related to the smooth category. However, it also introduced Freedman theory to show that the interior of every oriented 2-handlebody is {\em homeomorphic} (but frequently not diffeomorphic) to a Stein surface, and to construct uncountably many Stein diffeomorphism types homeomorphic to $\R^4$. Subsequently, two other papers spun off of this work. A necessary and sufficient condition was given in \cite{steindiff} for an open subset or embedded 2-handlebody in a complex surface to be smoothly isotopic to a Stein open subset or closed domain, with various applications connecting the smooth and analytic categories. (A related theorem is presented from a different viewpoint in \cite{CE}.) In \cite{MinGen}, smoothing theory for open 4-manifolds was explored, using Freedman's technology to construct smooth structures and distinguishing them by the genus inequality for Stein surfaces (ultimately from gauge theory). The present paper dives deeper into Freedman's work, but has been much simplified by the intervening publications.

According to Eliashberg, an almost-complex structure on a smooth $2n$-manifold $X$ is homotopic to a Stein structure if and only if $X$ is diffeomorphic to the interior of an $n$-handlebody (i.e., a smooth handlebody built with handles of index at most $n$, possibly infinitely many; see \cite{yfest} in the infinite case), subject to an additional constraint on 2-handle framings when $n=2$. In the latter case, the constraint vanishes if we work up to homeomorphism \cite{Ann}, but is necessary in the smooth setting: $S^2\times\R^2$ has no Stein structure. Note that almost-complex structures on oriented 2-handlebodies always exist since $\SO(2n)/{\rm U}(n)$ is 1-connected. Eliashberg's methods also apply up to isotopy for open subsets of complex manifolds. (See \cite{yfest} for a survey of both the abstract and embedded cases.) Recall that a {\em smooth\/} (resp.\ {\em topological\/}) {\em isotopy\/} is a $C^\infty$ (resp.\ $C^0$) homotopy through $C^\infty$ (resp.\ $C^0$) embeddings. An {\em embedding\/} is a (not necessarily proper) map that is a diffeomorphism (resp.\ homeomorphism) onto its image. An isotopy of an inclusion map is also called an isotopy from the domain to its image under the final embedding. An {\em ambient isotopy\/} is a homotopy of the identity map of the ambient space through diffeomorphisms (resp.\ homeomorphisms). According to \cite{steindiff}, an open subset $U$ of a complex surface is smoothly isotopic to a Stein open subset if and only if the inherited almost-complex structure on $U$ is homotopic to a Stein structure on it. (The corresponding statement is true in all dimensions by \cite[Theorem~13.8]{CE}, but \cite{steindiff} gives a stronger result for smoothly embedded 2-handlebodies up to smooth ambient isotopy; see Theorem~\ref{main} below.) The main theorem of \cite{JSG} essentially states that every topologically embedded 2-handlebody in a complex surface has interior topologically isotopic to a Stein open subset. The first main result of the present paper, Theorem~\ref{iso}, upgrades this to an ambient isotopy. This improvement lays the foundation for our subsequent results by creating a Stein neighborhood whose boundary is a topological 3-manifold. The following is a simplified version of Theorem~\ref{iso}:

\begin{thm}\label{iso1}
Let $X$ be a complex surface, and let $\mathcal{H}\subset X$ be a 4-dimensional, topologically embedded, collared 2-handlebody with $\cl\mathcal{H}-\mathcal{H}$ totally disconnected. Then there is a topological ambient isotopy sending the interior of $\mathcal{H}$ to a subset that is Stein in the complex structure inherited from $X$.
\end{thm}

\noindent The condition on the closure $\cl\mathcal{H}$ controls clustering of the handles of $\mathcal{H}$ inside $X$. In particular, it is vacuously true if $\mathcal{H}$ is a finite handlebody, or more generally if the embedding is proper. A (smooth or topological) {\em collar} on a codimension-1 submanifold $M$ is an extension of the inclusion to an embedding of $[0,1]\times M$ (where we identify $\{0\}\times M$ with $M$ in the obvious way). We call $M$ {\em bicollared} if it is collared on both sides. (This usually implies the topological category, since $M$ is always bicollared in the smooth, orientable setting.) Since the boundary of a topological manifold $Y$ is collared in $Y$, we call a codimension-0 submanifold {\em collared} if its boundary is bicollared. A codimension-0 submanifold becomes collared if we (nonambiently) isotope it along the boundary collar into its interior.

\begin{cor} \label{char} An open subset $U$ of a complex surface $X$ is topologically isotopic to a Stein open subset if and only if it is homeomorphic to the interior of a 2-handlebody.
\end{cor}

\noindent There may not be any ambient (or smooth) isotopy, as we see from the example $U=X= \CP ^1 \times \complex$.

\begin{proof}
If $U$ is topologically isotopic to a Stein surface $V$, then $U$ is obviously homeomorphic to $V$, which is diffeomorphic to a 2-handlebody interior (using \cite[Proposition~A.1]{yfest} if there are infinitely many handles). Conversely, the given handlebody $\mathcal{H}$ can be isotoped inside itself onto a closed subset of its interior. Restricting this isotopy to $\inter \mathcal{H}$ allows us to topologically isotope $U$ inside itself onto the interior of a closed subset of $U$ homeomorphic to $\mathcal{H}$. Apply Theorem~\ref{iso1} to $\mathcal{H}$ inside the complex surface $U$.
\end{proof}

The Stein surfaces arising from Theorem~\ref{iso1} are constructed by Freedman theory, so they typically have infinite smooth topology, even when the initial handlebodies are finite and smooth. Thus, they usually are not diffeomorphic to the interiors of compact 4-manifolds. Nevertheless, we have considerable control over their diffeomorphism types. Each is the interior of a smoothly embedded infinite 2-handlebody. We can control the combinatorics well enough to arrange the diffeomorphism types to be universal in the following sense: For all embeddings $\mathcal{H}\subset Y$ of a fixed $\mathcal{H}$ as in Theorem~\ref{iso1}, diffeomorphic in a neighborhood of $\cl\mathcal{H}$ to the given one into $X$ and with a suitable bound on the Chern classes $c_1(TY|\mathcal{H})$, the resulting Stein surfaces generated by the theorem are all diffeomorphic. (Corollary~\ref{diff} states this in a stronger form.) On the other hand, these universal types can be chosen flexibly: They come in infinitely many diffeomorphism types if ${\rm H}_2(\mathcal{H})\ne0$, and uncountably many under various hypotheses, for example if $X$ is $\C^2$ or a blowup of it. See Theorem~\ref{infDiff} and subsequent text.

The extra control due to the ambient nature of the isotopy in Theorem~\ref{iso1} allows us to find deeper structure in our Stein surfaces, spanning the topological and analytic categories. To understand this structure, recall that a smooth handlebody $\mathcal{H}$ whose handles all have index less than $\dim\mathcal{H}$ inherits a smooth mapping cylinder structure. That is, there is a smooth surjection $\psi\co[0,1] \times \partial \mathcal{H} \to \mathcal{H}$ sending $\{ 0 \} \times \partial \mathcal{H}$ onto the core $K$ of $\mathcal{H}$, a CW-complex with smooth cells, and restricting to the identity $\{ 1 \} \times \partial\mathcal{H}\to\partial\mathcal{H}$ and to a diffeomorphism $ (0,1] \times \partial \mathcal{H} \to \mathcal{H} - K$. (See the beginning of Section~\ref{Onionproof}.) The closed subsets $\mathcal{H}_\sigma = \psi([0,\sigma] \times \partial \mathcal{H}),\ 0<\sigma \le 1$, are all canonically diffeomorphic rel $K$, and their interiors form a neighborhood system of $K$ when $\mathcal{H}$ is a finite handlebody. (That is, when $\partial\mathcal{H}$ is compact, every neighborhood of $K$ contains some $\inter\mathcal{H}_\sigma$). For $\mathcal{H}$ topologically embedded in $X$ as in Theorem~\ref{iso1}, we can simultaneously control uncountably many of these interiors, namely those indexed by the standard Cantor set $\Sigma\subset[0,1]$.

\begin{de}\label{oniondef}
A topological embedding $g\co\mathcal{H}\hookrightarrow X$ will be called a {\em Stein onion on $\psi$ with core $g(K)$} if for each $\sigma\in\Sigma-\{0\}$,  the open set $g(\inter\mathcal{H}_\sigma)$ is Stein, $g|K$ is smooth except on one point of each open 2-cell, and the 2-cells are totally real (i.e.\ not tangent to any complex line) where they are smooth.
\end{de}

\begin{thm} \label{onion1}
Suppose a 2-handlebody $\mathcal{H}$ is given a mapping cylinder structure $\psi$ as above and a topological embedding as in Theorem~\ref{iso1} (i.e., it is collared with $\cl\mathcal{H}-\mathcal{H}$ totally disconnected).  Then the final embedding $g$ of the isotopy of Theorem~\ref{iso1} can be chosen to be a Stein onion on $\psi$.
\end{thm}

\noindent We will derive this from the stronger Theorem~\ref{onion} (\ref{onion2} if  $\mathcal{H}$ is infinite). When $\mathcal{H}$ is finite, the resulting 2-complex $g(K)$ is a {\em Stein compact}, i.e., a compact subset with a Stein neighborhood system. In fact, it inherits an uncountable system of nested Stein neighborhoods $g(\inter\mathcal{H}_\sigma)$ that are all topologically ambiently isotopic rel $g(K)$ (cf.\ Corollary~\ref{2cplx}). In general, the nonsmooth point on each 2-cell $D$ of $g(K)$ is very complicated. While such a point is undistinguished in the topological category, a typical $D$ has the property that every smooth local radial function on $X$ at the nonsmooth point must have infinitely many local minima on $D$ \cite{Prop2Knots}. By \cite[Corollary~7.11]{steindiff}, every Stein surface $V$ can be smoothly isotoped within itself to a Stein open subset that is the interior of a properly embedded 2-handlebody with strictly pseudoconvex boundary and totally real core. This allows us to directly construct a smooth open mapping cylinder structure on $V$ parametrized by $[0,\infty)$, restricting to a Stein onion on $[0,1]$. However, if the 2-complex $K\subset X$ is obstructed from being both smooth and totally real (either by unsmoothability or by an obstruction such as $c_1(X)$), then the full power of the theorem is required. In that case, the neighborhoods are unlikely to be diffeomorphic to $\inter\mathcal{H}$ or each other. In fact, we can guarantee (Section~\ref{Exotic}) that the Stein surfaces with $\sigma\in\Sigma-\{0\}$ realize infinitely many diffeomorphism types if ${\rm H}_2(\mathcal{H})\ne0$, and uncountably many under various hypotheses, for example, whenever $X$ is $\C^2$. However, we can still choose the resulting collection of diffeomorphism types universally (Addendum~\ref{univ}). A Stein surface $V$ obtained from a finite $\mathcal{H}$ by Theorem~\ref{iso1} will usually not have finite type, but in contrast to applying the direct smooth approach to $V$, Theorem~\ref{onion1} still gives a compact core with a Stein neighborhood system.

Theorem~\ref{onion1} allows us to more deeply explore topological notions related to pseudoconvexity in Section~\ref{Psc}. It is common to realize a Stein open subset as the interior of a compact region bounded by a smooth, strictly pseudoconvex 3-manifold. Such open subsets necessarily have finite smooth topology, unlike the examples produced by Freedman theory. However, our examples are bounded by topologically embedded 3-manifolds, which we call topologically pseudoconvex. More generally, we call a bicollared, topologically embedded 3-manifold $M$ {\em topologically pseudoconvex} (Definition~\ref{deftoppsc}) if it has a homeomorphism $h$ to the boundary of some Stein compact, with $h$ extending biholomorphically to a neighborhood of $M$. Such 3-manifolds inherit some of the basic properties of hypersurfaces that are strictly pseudoconvex in the usual sense (e.g., Proposition~\ref{pscbasics}). However, there is much more flexibility in the topological setting. For example, for $\Sigma$ the Poincar\'e homology sphere, $\Sigma\#\bar\Sigma$ admits no tight contact structure with either orientation \cite{EtH}, and hence, no strictly pseudoconvex embedding. However, it bounds a 2-handlebody $I\times(\Sigma-\inter B^3)$ that embeds in $\C^2$, so Theorem~\ref{onion1} gives a topologically pseudoconvex embedding in $\C^2$. In fact, every closed, oriented 3-manifold $M$ has a topologically pseudoconvex immersion in $\C^2$ and embedding in any closed, simply connected complex surface with sufficiently large $b_\pm$ (Theorem~\ref{toppsc}). We explore this phenomenon in more depth in \cite{TPC}: The homotopy class of the contact plane field on a smooth, strictly pseudoconvex 3-manifold has a well-defined generalization to the topological case (as a class of almost-complex structures on $\R\times M$). In the smooth setting, such classes are highly constrained. For example, only one of infinitely many homotopy classes on $S^3$ has a smooth, strictly pseudoconvex realization. However, every class on every $M$ is realized (for example) by a topologically pseudoconvex embedding in a closed, simply connected complex surface.

Cutting topologically pseudoconvex subsets out of a compact, complex surface leaves a topologically pseudocon{\em cave} region. In the smooth setting, such a region made by cutting along a strictly pseudoconvex $S^3$ is the complement of a ball in a compact, complex surface \cite[Theorem~16.5]{CE}, so cannot be contractible. However, we show that uncountably many exotic smoothings of $\R^4$ admit topologically pseudoconcave complex structures cut out by $S^3$ (Theorem~\ref{R4}(a)). That same theorem relates to an old problem in complex analysis: It is still an open question whether $\C P^2$ can be split into two Stein open subsets along a Levi flat 3-manifold. (The analogous question in higher dimensions has a negative answer \cite{S}.) Theorem~\ref{R4}(a) shows that $\C P^2$ can at least be split into two Stein open subsets along a (thick) topologically pseudoconcave embedded $I\times S^3$, and the same holds for every simply connected, compact, complex surface. Without simple connectivity, we still obtain such a splitting along some $I\times M^3$ (Corollary~\ref{split}).

The author's previous papers relating Stein surfaces to Freedman theory used {\em Casson handles}. These were developed by Casson~\cite{C} in the 1970s in an attempt to extend the powerful tool theorems of high-dimensional topology to dimension 4. The proofs of these theorems depended on being able to find certain embedded 2-disks. In dimension 4, the corresponding disks were generically only immersed with transverse double points. Casson extended these disks to infinite towers of immersed disks, whose neighborhoods are now called Casson handles. Freedman's fundamental breakthrough began with a proof that every Casson handle is homeomorphic to an open 2-handle $D^2\times \R^2$. This allowed high-dimensional topology to work for simply connected topological 4-manifolds, resulting in their complete classification when closed. The connection to Stein theory~\cite{Ann} is that while Stein structures do not always extend over 2-handles, they can always be extended by Eliashberg's method if double points are suitably introduced. Thus, replacing 2-handles by suitable Casson handles allows Stein structures without changing the underlying homeomorphism type. The diffeomorphism type typically changes. See~\cite{JSG} for a more leisurely discussion.

Casson handles have a technical drawback: their natural compactifications are not topological manifolds. Thus, their constructed embeddings are not ambiently isotopic to interiors of closed 2-handles, as would be required for Theorem~\ref{iso1}. This drawback also seriously complicated Freedman's original paper. Freedman eventually circumvented the difficulty by adding many layers of embedded surfaces between the layers of immersed disks, obtaining embeddings whose closures were collared topological 2-handles (e.g.,~\cite{FQ}). We will call the resulting objects {\em Freedman handles}. The use of Freedman handles significantly simplified his proof, and allowed application to many nonsimply connected 4-manifolds.

The main tool of the present paper is Theorem~\ref{AnnGCH}, that Stein structures can be extended over Freedman handles. This underlies Theorem~\ref{iso1}. It also brings Freedman's original 2-handle recognition proof within range of Stein theory, the result being Theorem~\ref{onion1}. Unfortunately, Stein theory requires our Freedman handles to be slightly different from the form appearing in \cite{FQ}, so for completeness, we develop the topological theory in sufficient generality in Section~\ref{GCH}. (This issue should also be rectified in the upcoming book \cite{R}.) In fact, we use extra generality for no additional cost, showing in passing that a large class of ``generalized Casson handles" are homeomorphic to open 2-handles (Corollary~\ref{homeo}). The methods of Section~\ref{GCH} and the nonanalytic framework of Section~\ref{Onionproof} are mainly due to Freedman, and the corresponding topological results will not surprise the experts. Freedman's original proofs were quite difficult, but fortunately his main results allow us to greatly simplify some arguments in the Stein version without circularity.

To summarize, we discuss basics of Freedman handles in Section~\ref{GCH}. In Section~\ref{Iso} we introduce Stein theory, prove Theorem~\ref{iso1} and discuss universal diffeomorphism types. We construct Stein onions in Section~\ref{Onionproof}, proving Theorem~\ref{onion1}, and analyze the resulting variety of diffeomorphism types in Section~\ref{Exotic}. Finally, Section~\ref{Psc} discusses topological pseudoconvexity and concavity. We work in the smooth category except where otherwise indicated. All manifolds are assumed connected (unless otherwise specified) and oriented, and complex structures and local homeomorphisms and diffeomorphisms are assumed to respect the given orientations. We abuse terminology by taking ``pseudoconvex" to mean ``strictly pseudoconvex", and similarly for ``plurisubharmonic". (The weaker notions never arise in this paper.)

The author wishes to thank Franc Forstneri\v c and the referees for helpful comments on the first version of the paper, and Yasha Eliashberg and Mike Freedman for decades of informative conversations.


\section{Generalized Casson handles} \label{GCH}

Throughout the text, we will be working with various smooth 4-manifolds $H$ that, like 2-handles, are meant to be attached to the boundary of another manifold along a (compact or open) solid torus $\partial_- H\subset\partial H$. Embeddings of such manifolds in each other are assumed to be diffeomorphisms on the attaching regions and (without loss of generality) preserve the central attaching circle. Such an embedding of $H$ will be called {\em collared} if $\partial H-\partial_-H$ is bicollared. We will call $(H,\partial_-H)$ an {\em open 2-handle homeomorph} if it has a homeomorphism to $(D^2,\partial D^2)\times\inter D^2$. After isotopy, we can assume the homeomorphism is a diffeomorphism on $\partial_-H$. A {\em closed 2-handle homeomorph} will be a topological pair $(H,\partial_-H)$ homeomorphic to a 2-handle $(D^2,\partial D^2)\times D^2$, with a smoothing on the preimage $H^o$ of $D^2\times\inter D^2$ (which is then an open 2-handle homeomorph). The attaching circle of a 2-handle homeomorph inherits a canonical framing of its normal bundle. More generally, if $C$ is a circle in the boundary of a 4-manifold $X$, and the boundary map induces an isomorphism ${\rm H}_2(X,C) \cong {\rm H}_1 (C) \cong \zed$, then there is a compact, embedded surface $F$ in $X$ with boundary $C$. Choose any framing on the normal bundle of $F$ and restrict to $C$. To see that this framing $\phi$ on $C$ is independent of $F$ and its chosen framing, simply add a 2-handle to $X$ along $C$ using the framing $\phi$ to get a 4-manifold $X'$ with ${\rm H}_2(X') \cong \zed$ and vanishing intersection form. For any other choice of $F$, its union with the core of the 2-handle must also be a surface with trivial normal bundle, so $\phi$ agrees with any framing on it. These canonical framings are clearly preserved by inclusions $C \subset Y \subset X$. (Frame $C$ in both manifolds using a surface in $Y$.) In the case of 2-handle homeomorphs, the canonical framing agrees with the product framing induced by the homeomorphism to $D^2 \times \inter D^2$, because the intersection form on homology is well-defined in the topological category.

Let $F$ be a compact (connected, orientable) generically immersed surface in a 4-manifold. Throughout the text, we denote the genus of $F$ by $g=g(F)$, and similarly use $k_+$ and $k_-$ to denote the numbers of its positive and negative double points, respectively. There are standard ways of increasing these numbers. To increase $g$ by one, ambiently take the connected sum with a small torus given by the standard model in $ \real ^3 \subset \real ^4$. There is an obvious pair of embedded disks in $\real ^3$ in the model, bounded by a standard basis for ${\rm H}_1 (T^2 )$, such that attaching 2-handles with these cores to a closed tubular neighborhood of the new surface $F'$ gives back a tubular neighborhood of $F$ (up to ambient isotopy). The framings for attaching this dual pair of 2-handles are uniquely characterized as being compatible with the projection of the tubular neighborhood to $F'$. There is also a standard model for increasing $k_+$ or $k_-$ by one (e.g.~\cite{FQ}, \cite{GS}). When $F$ is embedded, a closed tubular neighborhood of the resulting surface, up to diffeomorphism, is obtained from a disk bundle on $F$ by {\em self-plumbing\/}. That is, we restrict the disk bundle to two disjointly embedded disks in $F$, to get two copies of $D^2 \times D^2$, then identify these with each other so that the fiber directions on one correspond to the base directions of the other, and smooth the corners. When $\partial F$ is empty, the initial disk bundle is not the normal disk bundle $\nu F$, but differs from it by two twists, since the homological self-intersection number is unchanged by creating the double point, and is given for a generically immersed surface by
$$
 [F] \cdot [F]= e(\nu F)+2(k_+-k_-).
$$
(Each positive double point contributes two positive intersections of $F$ with its transverse pushoff defining $[F] \cdot [F]$.) For $(F,\partial F)$ immersed in $(X,C)$ as in the previous paragraph, capping $\partial F$ with a 2-handle as before shows more generally that the normal Euler number of $F$ relative to the canonical framing $\phi$ is $-2(k_+-k_-)$. As with increasing the genus, we can attach a 2-handle in the standard model  of an added double point to get back a tubular neighborhood of $F$. The  framing for the attached 2-handle is uniquely determined by this condition \cite{C}.

To build generalized Casson handles, we first need a basic building block. Let $F$ be a compact surface whose boundary is a circle, and let $T_1$ be obtained from $F \times D^2$ by performing some self-plumbings. We call $C=\partial F \times \{ 0 \} \subset T_1$ the {\em attaching circle\/} and $\partial _- T_1 =\partial F \times D^2$ the {\em attaching region\/}, and let $\partial_+ T_1 =\partial T_1 - \inter \partial _- T_1$. We call the image of $F \times \{ 0 \}$  the {\em core\/} of $T_1$, and continue to denote it by $F$. Since the boundary map is an isomorphism ${\rm H}_2(T_1 ,C) \cong {\rm H}_1 (C)$, the attaching circle has a canonical framing, which in general differs from the product framing induced from $F \times D^2$ (by $\pm2(k_+-k_-)$ twists, depending on our choice of orientation of the attaching region). There is a canonical projection from $T_1$ to its core. A framed link $L$ in $\partial _+T_1$ will be called a {\em trivializing link\/} if its projection to $F$ is nonseparating, and attaching 2-handles to $L$ using the given framings transforms $T_1$ diffeomorphically into a standard 2-handle (necessarily preserving the canonical framing on $C$). These 2-handles will be called {\em trivializing 2-handles}. A trivializing link or sublink, and its corresponding 2-handles, will be called {\em standard} if the 2-handles are given by the above standard models.

\begin{de} \label{gct}
The above $T_1$, with a preassigned choice of trivializing link, is a {\em 1-stage generalized Casson tower\/}. For $n>1$, an {\em $n$-stage generalized Casson tower\/} $T_n$ is obtained from an $(n-1)$-stage generalized Casson tower $T_{n-1}$ by adding an {\em $n^{th}$-stage\/} consisting of 1-stage towers attached along their canonically framed attaching circles to the trivializing links for the components of the $(n-1)^{st}$ stage. We call the union of these framed links the {\em trivializing link\/} for $T_{n-1}$. The {\em attaching circle\/} $C$ and {\em attaching region\/} $\partial _- T_n $ are those of $T_1$, and $\partial_+ T_n$ is given by $\partial T_n - \inter \partial _- T_n$ as before.
\end{de}

The requirement that the trivializing link of each 1-stage tower has nonseparating projection to $F$ is only needed for the most general cases of Theorems~\ref{onion} and \ref{onion2} (controlling inner layers for a Stein onion via Addendum~\ref{embFHAdd}(b), relative to a preassigned tower structure). Its main consequence is that every tower $T_n$ embedded in an open 4-manifold can then be identified as a smooth regular neighborhood of a 1-complex: For the first stage $T_1$, such a 1-complex is obtained from the projection of the trivializing link by adding arcs so that the complement in $F$ becomes an annulus. These new arcs connect the second stage of $T_2$ into a single 1-stage tower such that $T_1$ lies in a collar of the boundary. Continuing up the tower produces the required 1-complex with regular neighborhood $T_n$.

An {\em infinite\/} tower $T$ is a (noncompact) union $\bigcup_{n=1}^\infty T_n$ with $T_1 \subset T_2 \subset T_3 \subset\cdots$ as in Definition~\ref{gct}, with $\partial_\pm T$ as before. Since each 1-stage tower canonically embeds in a 2-handle as the complement of its trivializing 2-handles, every tower embeds in a 2-handle, canonically if the tower is finite. An infinite tower $T$ determines a based tree, where each vertex $v$ corresponds to a 1-stage tower in $T$, whose trivializing link has components corresponding to the outward edges from $v$.

\begin{de} \label{gch}
A {\em (closed) generalized Casson handle\/} $G$ is an infinite generalized Casson tower such that each infinite branch of the corresponding tree has infinitely many vertices corresponding to immersed disks. The subset $G^o=G-\partial_+ G$ is the corresponding {\em open generalized Casson handle\/}.
\end{de}

\noindent Although some of our theorems apply in this generality at no extra cost, it suffices to consider several special cases. Casson handles are the case in which all surfaces are disks and all trivializing links are standard.  These were Freedman's original tools, and were exclusively used in the author's previous work applying Freedman theory
 to Stein structures because of their combinatorial simplicity. In the present paper, we will mainly focus on {\em Freedman handles\/} (Definition~\ref{fh}). These include the infinite towers used in \cite{FQ}, which simplified and strengthened Freedman's work. In that text, each immersed disk was required to have $k_+=k_-$, which simplified the discussion of framings and allowed a slightly different notion of standard trivializing link for the corresponding 1-stage towers. As we will see, however, Stein theory requires disks with $k_+>k_-$, necessitating our present generalization of the theory.

Definition~\ref{gch} is motivated by the following proposition which, in the context of Casson handles, also motivated Casson's early work on the subject.

\begin{prop} \label{1conn}
Every generalized Casson handle $G$ is simply connected.
\end{prop}

\begin{proof}
 Every loop $\beta$ in $G$ has compact image, so it lies in some finite subtower $T_n$ of $G$. Now $T_ {n+1}$ is made from $T_n$ by adding 1-stage towers along its trivializing link $L$. Adding 2-handles to $T_n$ along $L$ yields a simply connected space (a 2-handle), so $\beta$ in $\pi_1(T_n)$ is a product of conjugates of circles of $L$ (suitably attached to the base point). If $\beta$ is nontrivial in $G$ then so is at least one such circle $c$ of $L$. In particular, the 1-stage tower attached to $c$ cannot be made from an immersed disk. Repeating the argument in the infinite subtower $T$ of $G$ attached along $c$, using nontriviality of $c$ in $T$, locates a component of stage $n+2$ whose core is not a disk. Iterating, we construct an infinite branch of the tree contradicting Definition~\ref{gch}.
\end{proof}

\noindent The definition could be weakened slightly, using the fact that double points in the surface attached to $c$ will not contribute to the expansion of $c$. However, the hypothesis cannot be eliminated, since {\em infinite gropes\/} (towers of embedded surfaces of positive genus) are known not to be simply connected.

One can actually make a much stronger assertion about open generalized Casson handles: They are always homeomorphic to open 2-handles (although typically not  diffeomorphic, due to Donaldson's work). This assertion can be proved using Freedman's topological proper h-Cobordism Theorem, although our methods below generate a proof more closely related to Freedman's original approach (Corollary~\ref{homeo}). For now, we only need a weaker statement:

\begin{lem} \label{2h}
Every generalized Casson handle $G$ contains a topologically embedded, collared 2-handle $h_G$ such that $\pi_1 (G-h_G) \cong \zed$, generated by a meridian $\mu$ of its attaching circle $C$.
\end{lem}

To prove this, we invoke Freedman's fundamental Disk Embedding Theorem \cite{FQ}, the successor to his original theorem that Casson handles are 2-handle homeomorphs. Given a collection of disks immersed in a 4-manifold $X$, with their boundaries embedded in $\partial X$, the Embedding Theorem gives conditions guaranteeing the existence of disjointly embedded topological 2-handles whose cores have the same framed boundaries. The fundamental group of $X$ must be ``good'' (e.g.\ abelian), and there must be dual spheres with suitable properties. Setting this up in $G$ requires some work involving Freedman's (and Casson's) original techniques.

\begin{proof}
For each $n$, let $T_n$ denote the first $n$ stages of $G$, and let $T'_n \subset T_n$ be obtained by removing a collar of $\partial_+ T_n$ (so that $T'=\bigcup_{n=1}^\infty T'_n$ is an infinite tower isotopic to $G$). Our first task is to show that $\pi _1 (G-T'_n ) \cong \zed$. But $G$ is obtained from $T_n$ by attaching generalized Casson handles to its trivializing link $L$. Since these generalized Casson handles are simply connected, the fundamental group is unchanged if we replace them by 2-handles, transforming $G$ into a 2-handle $h_n$. That is, $\pi _1 (G-T'_n ) \cong \pi _1 (h_n -T'_n )$. But $h_n -T'_n$ is obtained from $T_n - T'_n = (0,1] \times \partial_+ T_n$ by adding 2-handles along $L$. Looking at this relative handlebody upside down identifies it as $[0,1] \times \partial_+ h_n$ union 2-handles (with some boundary removed). Since $\pi_1 (\partial_+ h_n) \cong \zed$, the desired group $\pi_1 (G-T'_n) $ must be a quotient of $\zed$, generated by $\mu$. But $\mu$ cannot have finite order $k$; otherwise a nullhomology of $k\mu$ could be combined with $k$ disks intersecting $C$ to create a surface in $h_n$ intersecting the core of $T_1$ $k \ne 0$ times, contradicting triviality of ${\rm H}_2 (h_n )$. (Recall the relative intersection pairing ${\rm H}_2 (h_n ) \times {\rm H}_2 (h_n ,\partial h_n ) \to \zed $.) Thus $\pi _1 (G-T'_n ) \cong {\rm H}_1 (G-T'_n ) \cong \zed$, generated by $\mu$, or more generally by any meridian of the first stage core (since these are all homologous up to sign). If $n>1$ and $\mu'$ is any meridian of a second stage core $F$, then it is nullhomologous and hence nullhomotopic in $G-T'_n$, since a 2-sphere centered at a point on $\partial F$ provides a homology from $\mu'$ to a pair of oppositely oriented meridians of the first stage (Figure~\ref{sphere}). Similarly, all higher stage meridians are nullhomotopic.

\begin{figure}
\labellist
\small\hair 2pt
\pinlabel $F$ at 82 110
\pinlabel $\mu$ at 9 54
\pinlabel $-\mu$ at 141 54
\pinlabel $\mu'$ at 99 88
\pinlabel $T'_1$ at 161 27
\endlabellist
\centering
\includegraphics{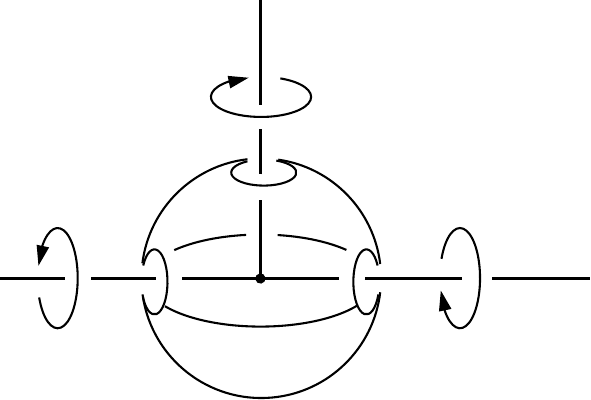}
\caption{A nullhomology of a second-stage meridian. (Here $F$ and $T'_1$ extend in the fourth direction.)}
\label{sphere}
\end{figure}
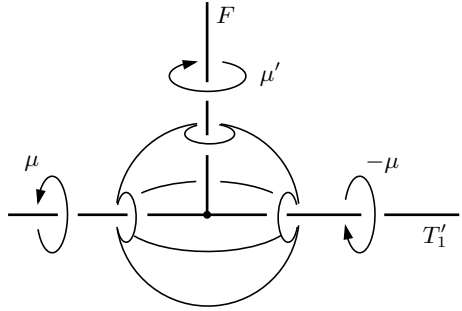

We will apply Freedman's Embedding Theorem, starting in the disconnected manifold $X=G-(T_1 \cup \inter T'_2)$, and using the trivializing link $L$ for $T'_2$ in $\partial X$. Since each component $C_i$ of $L$ is nullhomotopic in the attached generalized Casson handle of $T'$, it bounds an immersed disk $D_i$ whose interior lies in $T'-T'_2 \subset X$. We add double points so that $k_+(D_i)=k_-(D_i)$. By compactness, the family $\{ D_i \}$ lies in $T'_n$ for some $n$. A meridian to each $C_i$ is nullhomotopic in $X-\inter T'_n \subset X- \bigcup D_i$ (by the previous paragraph applied to the disjoint union of generalized Casson handles $G-T_1$). Combining this nullhomotopy with a meridian disk for $C_i$, we obtain a sphere $S_i$ immersed in $X$, intersecting $D_i$ once (transversely) and disjoint from every other $D_j$. Thus $\{ S_j \}$ is a collection of {\em transverse spheres\/} for the disks $\{D_i \}$ as required for the Embedding Theorem. We can assume each $k_+(S_i)=k_-(S_i)$, but need to scrutinize the intersections among the spheres. Since $G$ embeds in a 2-handle (whose intersection form is identically zero), all of the intersection numbers $S_i \cdot S_j$ vanish. Extra care is required here, though, since the hypotheses of the Embedding Theorem refer to more subtle intersection numbers with coefficients in $\zed [\pi _1 (X)]$. These involve classes in $\pi _1 (X) \cong \zed $ realized by loops in the image of the immersion of $\bigcup S_j$. (These loops can change sheets at the double points, and may be nontrivial in $\pi _1 (X)$.) We now pass to the larger space $G- \inter T'_2$ in which all these loops are trivial by the previous paragraph. The required intersection numbers now vanish in $\zed [\pi _1 (G- \inter T'_2)]$, so the Embedding Theorem applied to $\{ D_i \} $ in $G- \inter T'_2$  yields a topologically embedded collection $\{ h_i \} $ of 2-handles in $G- \inter T'_2$ attached to $L$ with the correct framings. (Since each $k_+(D_i)=k_-(D_i)$, we can ignore the distinction between the canonical framing and that induced by the normal bundle of $D_i$.) We can assume each $h_i$ is collared (after replacing it by a thinner handle if necessary). Then $h_G=T'_2 \cup  \bigcup h_i$ is our required topological 2-handle. The Embedding Theorem guarantees transverse spheres for the handles $h_i$ (see \cite{R}), so their meridians are nullhomotopic in $G-h_G$. Thus $\pi_1 (G-h_G) \cong \pi_1 (G-T'_2) \cong \zed$, generated by $\mu$ as required.
\end{proof}

We next examine when a generalized Casson handle can be smoothly embedded in a 4-manifold so that its closure is a topologically embedded 2-handle. This will lead to the main goal of the section, finding such an embedding in an arbitrary 2-handle homeomorph.

\begin{de} \label{conv}
An embedding $G \subset X$ of a generalized Casson handle into a 4-manifold is called {\em convergent\/} if for each $\epsilon >0$ there is an $n$ for which the components of $G-T_n$ lie in the interiors of disjoint closed topological balls of diameter $<\epsilon$ in $X$.
\end{de}

\noindent It follows that such a $G$ has compact closure in $X$, so the definition is independent of the choice of distance function. In fact, the closure is given by the endpoint compactification of $G$, so that $\cl G -G$ is in bijective correspondence with the infinite branches (starting at the base point) of the tree for $G$. We denote the endpoint compactification of a generalized Casson handle $G$ by $\bar G$, and set $\partial_+ \bar G =\bar G-G^o$ and $\partial\bar G=\partial_+\bar G\cup \partial_- G$. These need not be manifolds in general, even topologically.

\begin{de} \label{fh}
A generalized Casson handle $G$ will be called a {\em Freedman handle\/} if $\bar G$ is homeomorphic to a 2-handle. It will be called {\em standard} if for some $n$, each core surface beyond the $n^{th}$ stage is either embedded or a disk, and all trivializing links beyond the $n^{th}$ stage are standard.
\end{de}

\begin{prop} \label{fh1}
\rm a) There is a sequence $(a_k)$ of nonnegative  integers such that a standard generalized Casson handle $G$ must be a Freedman handle provided that for each infinite branch of the corresponding tree and each $k \in \zed^+$, the $k^{th}$ disk along the branch is separated from the next disk by at least $a_k$ vertices corresponding to embedded surfaces.
\item[b)] Let $G \subset X$ be a convergent Freedman handle with $\bar G\cap \partial X=\partial_- G$. Then $\bar G$ is collared if and only if for each $x \in \bar G -G$, every neighborhood $U$ of $x$ in $X$ contains a neighborhood $V$ for which inclusion $V-\bar G \subset U-\bar G$ is $\pi_1$-trivial.
\item[c)]  Every generalized Casson handle $G$ inherits a quotient map $\partial h\to\partial\bar G$, where $h$ is a 2-handle, that restricts to a diffeomorphism of the attaching regions.

\end{prop}

\noindent A sequence $(a_k)$ has the above property if and only if $\sum a_k/2^k$ is infinite \cite{AS}. It can begin with an arbitrarily long string of zeroes since every finite generalized Casson tower can be topologically trivialized by attaching Freedman handles, but eventually embedded surfaces must dominate the construction. In particular, Casson handles, built entirely from disks, are not Freedman handles.

\begin{proof}
We first prove (c) and that the condition of (a) implies $\partial_+ \bar G$ is homeomorphic to $D^2 \times S^1$. The proof of this is essentially contained in  \cite[4.2]{FQ} (particularly the lemma); see also \cite{AS}. Attaching trivializing 2-handles to the $n$-stage subtower $T_n$ of $G$ yields a 2-handle $h_n \approx D^2\times D^2$ with $\partial_+ h_n \approx D^2 \times S^1$. The new 2-handles intersect $\partial_+ h_n$ in an embedded collection of solid tori whose complement is $\partial_+ T_n$ minus a neighborhood of its trivializing link. We obtain each collection of solid tori by an iterated doubling procedure, which is easiest to see when all trivializing links are standard: For $T_1$, we first introduce disjoint solid tori parallel to $\{ 0 \} \times S^1 \subset D^2 \times S^1$, one for each double point or torus summand of the core, then replace each by its Whitehead double (for a double point) or Bing double (for a torus summand) as in Figure~\ref{double}. (For a negative double point, use the mirror image of Figure~\ref{double}(a).) Each time we add a 1-stage tower to $T_n$, we repeat the previous procedure on the corresponding solid torus, first passing to parallel solid tori inside the given one, then suitably doubling (see also \cite[Section~6.1]{GS}). Ultimately, we obtain an infinite nested collection of unions of solid tori in $D^2 \times S^1$ whose intersection $Q$, a {\em Bing--Whitehead compactum}, has complement diffeomorphic to $\partial_+ G$. Endpoint compactifying adds a point for each component of $Q$, so $\partial_+ \bar G$ is homeomorphic to the quotient space obtained from $D^2 \times S^1$ by collapsing each component of $Q$ to a point. By \cite{FQ} (see also \cite{AS}), if $G$ respects a sequence $(a_k)$ that grows sufficiently rapidly, then we can arrange the Bing doubles inside their solid tori in such a way that $Q$ is totally disconnected, so no collapsing is necessary and the quotient map is a homeomorphism. (In contrast, a Casson handle involves only Whitehead doubling, so the resulting Whitehead continuum $Q$ will be 1-dimensional, and $\partial_+ \bar G$ will not be a manifold.) If the trivializing links of $G$ only become standard after the $n^{th}$ stage, the solid tori for $T_n$ will be more complicated, but the previous argument applies inside each of them. When $G$ is any generalized Casson handle, the same argument still exhibits $\partial G$ as the quotient of $\partial h$ by collapsing an infinite nest of unions of solid tori.

\begin{figure}
\labellist
\small\hair 2pt
\pinlabel a) at -5 5
\pinlabel b) at 167 5
\endlabellist
\centering
\includegraphics{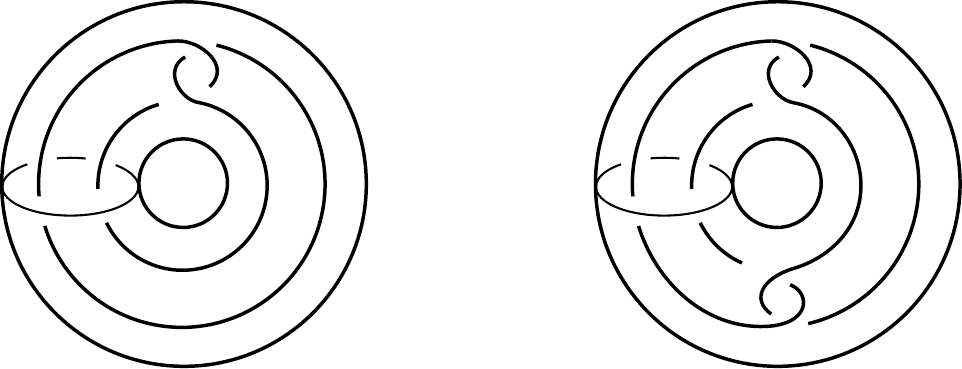}
\caption{Whitehead (a) and Bing (b) doubling.}
\label{double}
\end{figure}

Now suppose $G\subset X$ is a convergent generalized Casson handle with $\bar G\cap \partial X=\partial_- G$, and $\partial_+\bar G$ is a topological manifold. We determine when $\partial_+\bar G$ is bicollared, using \cite[9.3A]{FQ}. Essentially, we must collar the two sides of $\partial_+ \bar G$ separately by verifying a local condition at each $x \in \partial_+ \bar G$. The condition for collaring the outside is restated in (b) of the proposition, and is automatically satisfied at each $x \in \partial_+ G$, implying (b).  The analogous condition for collaring the inside is automatically satisfied for any $\bar G$: For $x \in \bar G -G$, we can always choose $V$ so that it intersects $G^o$ in a generalized Casson handle interior, which is simply connected. Note that by \cite{FQ}, each collar on $\partial_+\bar G$ extends a preassigned collar on $\partial X$.

It remains to complete the proof of (a). Under its hypotheses, we have already identified $\partial \bar G$ with $S^3$ with unknotted attaching circle, and collared it inside $\bar G$. It follows that $\bar G$ is a compact topological manifold with boundary $S^3$. But $\inter \bar G = \inter G$ is contractible (since $\pi_1(G)$ and ${\rm H}_*(G)$ are trivial), so by Freedman $\bar G$ is homeomorphic to a 4-ball $D^2 \times D^2$, and we can assume the attaching regions correspond.
\end{proof}

Our main theorem of this section produces a Freedman handle inside any 2-handle homeomorph. It is an application of Freedman's work that was essentially known to Quinn in the 1980s; a slightly weaker version for Casson handles is given in Quinn's paper \cite[Proposition~2.2.4]{Q}. See also \cite[Theorem~5.2]{JSG} for a more leisurely exposition of the latter. We will call a generalized Casson handle a {\em refinement} of another such $G$ if it is obtained from $G$ by adding genus and double points to the surfaces of $G$, then attaching generalized Casson handles. In the contexts of Freedman handles (standard or otherwise) we refine preserving the given context.

\begin{thm}\label{embFH} Every closed 2-handle homeomorph $H$ contains an embedded, convergent Freedman handle $G$ whose closure is topologically ambiently isotopic rel $\partial H$ to a standard topological product neighborhood $h_0$ of the topological core of $H$. The set of such $G$ (for fixed $H$) is closed under refinement, and contains Freedman handles that are standard at every stage, as well as refinements of any preassigned Freedman handle. If $H$ is a standard 2-handle, then every Freedman handle has such an embedding.
\end{thm}

\begin{Remark}\label{embFHRem} By the same proof and Corollary~\ref{homeo}, we can similarly find refinements of any generalized Casson handle, embedding with interior topologically isotopic (not ambiently) to $\inter h_0$.
\end{Remark}

\begin{proof}
We first construct a 2-stage tower in $H$ whose first stage is smooth and whose second stage consists of topologically embedded 2-handles. If $H$ is standard, we already have this with smooth 2-handles. Otherwise, our main tool is Quinn's Handle Straightening Theorem \cite[2.2.2]{Q}, which addresses when a homeomorphism from a standard $k$-handle can be smoothed near its core $D$ by a topological ambient isotopy. This can always be done when $k\le 1$.  When $k=2$, the homeomorphism can be smoothed near the image of $D$ under a smooth regular homotopy in the domain 2-handle. This homotopy consists of {\em finger moves} whereby we move a small disk of $D$ along an arc intersecting $D$ only at its endpoints, creating a pair of intersections near the other endpoint of the arc. Since $\pi_1(H-D)\cong\zed$ and homotopy implies isotopy for arcs in 4-manifolds, there is a smooth isotopy sending each finger move to a standard model. This consists of a pair of added double points of opposite sign, which have standard trivializing disks. The resulting image in $H$ is the required 2-stage tower, whose first stage surface $F$ is a smoothly immersed disk with $k_+=k_-$. We can add additional double points to independently increase $k_\pm$ as much as desired. Alternativaly, we can lower $k_\pm$ by smoothing double points, increasing $g$ by the number of smoothed double points. The original trivializing disk of each smoothed double point now attaches with the standard framing to one member of a dual pair of trivializing circles. The other circle bounds a disk created by the smoothing, but whose framing is not standard. However, since one 2-handle is correctly framed, the other framing can be adjusted arbitrarily by sliding the second handle over the first. This is most easily seen via Kirby calculus (e.g.~\cite{GS}). A torus summand of $F$ is represented as a Bing link of dotted circles (1-handles) around the attaching circle as in Figure~\ref{framing}(a)
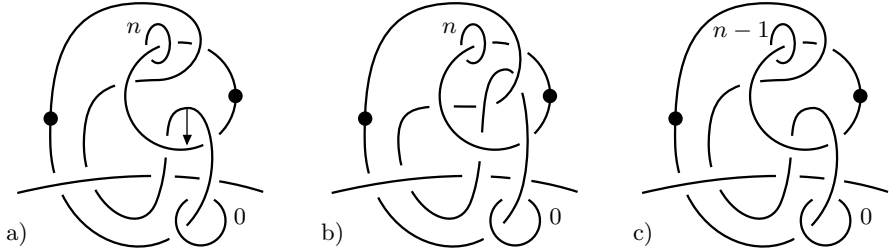
\begin{figure}
\labellist
\small\hair 2pt
\pinlabel a) at 0 5
\pinlabel b) at 119 5
\pinlabel c) at 236 5
\pinlabel 0 at 84 12
\pinlabel 0 at 203 12
\pinlabel 0 at 320 12
\pinlabel $n$ at 44 83
\pinlabel $n$ at 163 83
\pinlabel $n-1$ at 273 82
\endlabellist
\centering
\includegraphics{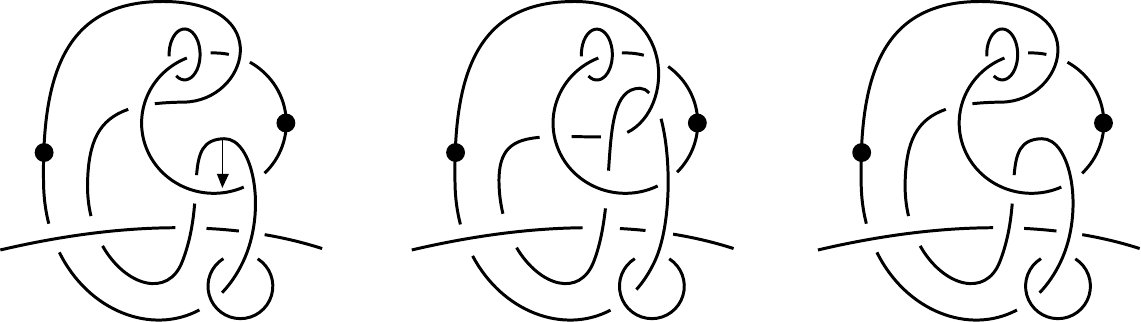}
\caption{Reframing a trivializing 2-handle.}
\label{framing}
\end{figure}
(e.g.~\cite[Section~6.2]{GS}). In this figure, the two meridians represent a standard pair of components of the trivializing link. One meridian has framing 0 (the framing compatible with the normal bundle projection); the other framing differs from the compatible one by an arbitrary integer $n$. If we slide 1-handles as indicated by the arrow, then the resulting dotted circle will have a clasp in it as in (b), and will link the $n$-framed circle. However, we can slide that circle over the 0-framed meridian to remove the latter linking (cf.~\cite[Proposition~5.1.4]{GS}) to obtain (b). (This is the required slide of trivializing disks, and the 1-handle slide changes the parametrization of the torus summand to exhibit the new dual pair as the standard basis for it.) Rotating one foot of the 1-handle given by the round dotted circle (as in \cite[Figure~5.42]{GS}) transfers the left twist from the other dotted circle to the $n$-framed meridian to yield (c). (This rotation is basically a change in the trivialization of the normal bundle to $F$.) Note that (c) differs from (a) only in the framing of one 2-handle. Since the process is reversible, we can inductively change any framing on the second 2-handle to the required compatible one. (A harder unpublished argument of Quinn is more efficient, replacing an entire finger move by a punctured torus.)

We now create a Freedman handle by induction. For the induction hypotheses, assume we have towers $T_n\subset T_{n+1}^* \subset H$, where $T_n$ is smooth and $T_{n+1}^*$ is obtained from it by adding topologically embedded 2-handles along a trivializing link. Also assume that the topological 2-handle $T_{n+1}^*$ is topologically ambiently isotopic  to $h_0$ rel $\partial H$. We already have this when $n=1$. For general $n\ge 1$, apply the previous paragraph to each of the top stage topological 2-handles, reproducing the induction hypotheses with $T_n\subset T_{n+1} \subset T_{n+2}^* \subset T_{n+1}^*$. We obtain an infinite tower $G=\bigcup_{n=1}^\infty T_n$ that we can choose to be a Freedman handle by requiring enough embedded surface stages as in Proposition~\ref{fh1}(a), although it may not be convergent. To obtain a refinement of a particular Freedman handle $G_0$, replace each disk $D$ in the construction by the corresponding surface of $G_0$ with its defining immersion in a 2-handle, smooth near its 1-skeleton by Quinn's theorem, and apply the previous paragraph to a 2-handle neighborhood of the remaining disk.

We next pass to the topological category and show that the 2-handle $h_G \subset G$ of Lemma~\ref{2h} is ambiently isotopic to $h_0$ in $H$, rel $\partial H$. Since $h_G$ is collared, $\cl(H- h_G)$ is a topological 4-manifold whose boundary is canonically homeomorphic to that of $\cl(H- h_0)$. It now suffices to show that $\pi_1 (H -h_G)\cong\zed$, for then a standard argument using the $s$-Cobordism Theorem with $\pi_1\cong \zed$ (e.g., \cite{FQ}) identifies $\cl(H- h_G)$ with $\cl(H- h_0)\approx S^1\times D^3$, and the induced self-homeomorphism of the 4-ball $H$ $\rel \partial H$, sending $h_G$ onto $h_0$, is isotopic to the identity (e.g.~by the Alexander Trick). Since ${\rm H}_1(H-h_G) \cong\zed$ (as is true for any 2-handle in $B^4$, e.g., by Mayer--Vietoris), it suffices to show that every nullhomologous loop $\gamma$ in $H-h_G$ is nullhomotopic. Since $h_G$ is a compact subset of $G$, it lies in some finite subtower $T_n$ of $G$. By construction, $T_n$ lies in $T_{n+1}^*$, which is unknotted, i.e., ambiently isotopic to $h_0$. Thus, we can assume $\gamma$ lies in $T_{n+1}^* -h_G$. Now we push $\gamma$ off of the top-stage embedded 2-handles of $T_{n+1}^*$, so that $\gamma$ lies in $T_n-h_G\subset G-h_G$. But inclusion $G-h_G \hookrightarrow H-h_G$ induces an ${\rm H}_1$-isomorphism, since both groups are $\zed$ generated by a meridian of $h_G$. Since $\gamma$ is nullhomologous in $H-h_G$, it is then nullhomologous in $G-h_G$. But Lemma~\ref{2h} guarantees that $\pi_1 (G-h_G) \cong \zed$, so $\gamma$ is nullhomotopic as required.

To arrange convergence of the Freedman handle, we return to the first paragraph. Starting from the surface $F$ obtained there, we choose a family of disjoint smooth balls in $H$ corresponding bijectively to components of its trivializing link, and construct a smooth ambient isotopy after which each trivializing 2-handle $h$ of the 2-stage tower can be assumed to lie in the corresponding ball: If $h$ is smooth, it is easy to do this. For the topological case, we can now construct a standard Freedman handle $G_h$ inside $h$. By the previous paragraph we can assume, after topologically narrowing $h$, that it lies inside $G_h$. By compactness, $h$ then lies inside some finite subtower $T_h$ of $G_h$, which can be identified as a smooth regular neighborhood of a wedge of circles by the text following Definition~\ref{gct}. Since smooth 1-complexes cannot knot or link in the simply connected 4-manifold $H$, each $T_h$ can now be smoothly isotoped into its corresponding ball. Applying this refined version of the first paragraph during the construction of $G$, we easily obtain a convergent Freedman handle.

Finally, we show that $\bar G$ is topologically ambiently isotopic to $h_G$, and hence to $h_0$. Since $\pi_1 (G-h_G) \cong \zed$ by Lemma~\ref{2h}, our previous reasoning shows that the pair $(\bar G,h_G)$ is homeomorphic to a 2-handle with a standard neighborhood of its core, so it suffices to show that $\bar G$ is collared in $H$. By construction, the trivializing handles comprising the top stage of each $T_n^*$ have maximal diameter approaching 0 as $n \to \infty$, so $\bar G=\cl G = \bigcap_{n=1}^\infty T_n^*$. Each $x\in \bar G-G$ has a neighborhood system $\{\inter h_n|n\in \zed^+\}$, where $h_n$ is a top-stage 2-handle of $T_n^*$. By Proposition~\ref{fh1}(b), it now suffices to show that inclusion $h_{n+1}-\bar G \subset h_n-\bar G$ is always $\pi_1$-trivial. But any loop $\gamma$ in $h_{n+1}-\bar G$ is compact, so it lies in $h_{n+1}-T_N^*$ for some $N$. This latter space is homeomorphic to $H-h_0$ by construction, so its fundamental group is $\zed$, and $\gamma$ is homotopic in $h_{n+1}-T_N^*$ to a power of the meridian $\mu'$ of the attaching circle of $h_{n+1}$. Similarly, $\pi_1(h_n-T_N^*)\cong \zed$, and in that space $\mu'$ is nullhomologous (cf.\ Figure~\ref{sphere}), hence nullhomotopic. Thus $\gamma$ is nullhomotopic in $h_n-\bar G$ as required.
\end{proof}

\begin{adden}\label{embFHAdd} In the previous theorem, we can also assume both of the following:
\begin{itemize}
\item[(a)] For a preassigned metric on $H$ and $\epsilon>0$, $G$ can be arranged so that for every $n\in\zed^+$, the components of $\bar G-T_n$ lie in disjoint smooth balls of diameter less than $\epsilon/2^n$.
\item[(b)] There is a collared topological 4-ball $B$ lying in a preassigned smooth ball in $H$ such that $G-\inter B$ is a smooth tubular neighborhood of an embedded annulus lying in the first stage core. This can be arranged, preserving a preassigned subtower $T_n$ of $G$ up to smooth isotopy supported in $\inter H$.
\end{itemize}
\end{adden}

Note that (b) is nontrivial since collared topological balls cannot usually be smoothly isotoped into small neighborhoods; many exotic smoothings of $\R^4$ contain counterexamples.

\begin{proof} 
Part (a) is immediate from the proof of the previous theorem. For (b), arrange the top stage of $T_{n+1}^*$ to lie in the smooth ball. Shrink $T_n$ into $\inter H$ by removing a collar of $\partial_-T_1$. Then $T_n$ is a regular neighborhood of a 1-complex that we can smoothly isotope into the ball. Let $B$ be the resulting image of $T_{n+1}^*$, which is still connected to the attaching circle of $H$ by an annulus.
\end{proof}


\section{Stein Freedman handlebodies}\label{Iso} 

We are now ready to adapt Freedman handles to Stein theory. After assembling our basic definitions and tools, we will strengthen the theorem from \cite{Ann} that every 2-handlebody interior is homeomorphic to a Stein surface. That theorem was proved by replacing 2-handles with Casson handles; Theorem~\ref{AnnGCH} replaces them with Freedman (or other generalized Casson) handles. This allows us to amplify \cite{JSG}, with Theorem~\ref{iso} showing that topologically embedded 2-handlebodies are {\em ambiently} isotopic to Stein Freedman handlebodies. We will use this in Section~\ref{Onionproof} to create Stein onions. For completeness, we keep track of which generalized Casson handles can arise in these constructions, with a sequence of addenda and remarks that are not central to the paper.

We begin with the tools we will need from Stein theory; see \cite{CE} for more details. Recall that as in \cite{CE}, we suppress the term ``strict" throughout the discussion. Stein surfaces are characterized as being open complex surfaces $(W,J)$ admitting plurisubharmonic Morse functions that are {\em exhausting}, i.e., proper and bounded below. {\em Plurisubharmonic} (equivalently {\em $J$-convex}) essentially means the level sets are {\em pseudoconvex} ({\em $J$-convex}) 3-manifolds away from the critical points. In this dimension, the latter is equivalent to saying that on each of these 3-manifolds $M$, the unique field of complex lines $TM\cap JTM$ is a contact structure. The Morse function then describes $W$ as a sequence of disjoint collars of the form $I\times M$ with each level $\{t\}\times M$ pseudoconvex, and each pair of consecutive collars separated by an elementary cobordism, essentially a layer of disjoint handles. Eliashberg reverses this description to build Stein manifolds. Theorems~8.4 and 8.5 of \cite{CE} give suitable hypotheses for attaching an elementary cobordism to the boundary of a complex $W$, preserving pseudoconvexity. Proposition~10.10 of \cite{CE} allows the resulting critical points to move vertically within the cobordism. Thus, we can interpret the procedure as attaching handles directly to a pseudoconvex boundary $M$, rather than adding an elementary cobordism so that $M$ lies in the resulting interior. (Extend the cobordism below $M$, push the critical points below $M$, then cut off remaining boundary above the level of $M$.) The new boundary then has a collar with pseudoconvex levels that can be assumed to agree with a preassigned pseudoconvex collar of the old boundary away from the new handles. (We always assume handles are attached smoothly so that the new and old boundaries are $C^\infty$-close where they join along the boundary of the attaching region. To get the new collar to match the old one across its entire $I$-parameter, first extend the original collar deeper into $W$ near the attaching regions, then push the critical points below the original collar and extend the latter using the associated plurisubharmonic function.) These theorems from \cite{CE} are stated in the context of compact manifolds, but the proofs are local (affecting only a small neighborhood of the core disks), so they can also be applied to noncompact boundaries with properly embedded collars and infinite collections of handles. This allows us to change our viewpoint to that of \cite{steindiff}, attaching our handles in just three layers, of index 0, 1 and 2, respectively. Since each compact subset of our manifold can only intersect finitely many handles, infinite topology will require infinitely many 0-handles or attaching to a preexisting manifold with noncompact boundary. Following \cite{steindiff}, we have:

\begin{de} For a $4$-manifold $W$ with boundary (not necessarily compact), a {\em 2-handlebody relative to} $W$ is a nested collection of $4$-manifolds $W=\mathcal{H}_{-1}\subset \mathcal{H}_0\subset \mathcal{H}_1\subset \mathcal{H}_2=\mathcal{H}$, where for each $k\ge0$, $\mathcal{H}_k$ is made from $\mathcal{H}_{k-1}$ by attaching (possibly infinitely many) $k$-handles with disjoint attaching regions. A {\em 2-handlebody} is the special case when $W$ is empty. A {\em 2-handlebody pair (relative to $W$)} is a pair $(\mathcal{H},\mathcal{H}^-)$  where $\mathcal{H}$ is a handlebody (relative to $W$) and $\mathcal{H}^-\subset\mathcal{H}$ is a {\em subhandlebody (relative to $W$)}, i.e., a subset that inherits a handlebody structure (relative to $W$). 
\end{de}

\noindent Since there are no 3-chains, $\rm H_2(\mathcal{H},\mathcal{H}^-)$ is free abelian, and when $W$ is empty, $\rm H_2(\mathcal{H}^-)$ inside $\rm H_2(\mathcal{H})$ is a direct summand of a free abelian group. We will need $W$ nonempty in Section~\ref{Onionproof}.

\begin{de}\cite{steindiff}. \label{SHB} A complex surface with boundary, exhibited as a 2-handlebody $\mathcal{H}$ (relative to a submanifold $W$), will be called a {\em Stein handlebody (relative to $W$)} if 
\item{a)} it admits a proper holomorphic embedding into a Stein surface,
 \item{b)} each $\partial \mathcal{H}_k$ ($k\ge -1$) is (strictly) pseudoconvex, and
 \item{c)} the attaching region of each 2-handle, which is a solid torus in the contact 3-manifold $\partial \mathcal{H}_1$, has convex boundary (in the contact sense) with  two dividing curves, each a longitude that when modified by a left twist determines the 2-handle framing.
\end{de}

\noindent The first two conditions guarantee \cite{steindiff} that every relative Stein handlebody interior is a Stein surface, and every compact relative Stein handlebody is a Stein domain (sublevel set of an exhausting plurisubharmonic function). The set of (relative) Stein handlebodies is closed under attaching handles by Eliashberg's method (suitably defined) \cite[Proposition~7.7]{steindiff}. Additional handles attached to $\mathcal{H}$ can be assumed to be attached by this method, provided that the new 2-handles have sufficiently negative framings. Condition (c), which we will not use in detail, expresses the essence of a 2-handle attached in this manner (determining a Legendrian attaching circle for a totally real core \cite[Proposition~7.6]{steindiff}). Note that each subhandlebody of a Stein handlebody is again a Stein handlebody.

We find embedded Stein handlebodies using the following theorem, essentially \cite[Theorem~7.9]{steindiff} together with the second sentence of its proof.

\begin{thm}\cite{steindiff}.\label{main}
Let $\mathcal{H}$ be a Stein handlebody relative to $W$, and let $f\co \mathcal{H}\to X$ be a smooth, proper embedding into a complex surface. Suppose that the complex structure on $X$ pulls back to one on $\mathcal{H}$ that is homotopic rel $W$ (through almost-complex structures) to the original complex structure. Then $f$ is smoothly ambiently isotopic to $\hat f$, with $\hat f(\mathcal{H})$ exhibited as a Stein handlebody relative to $\hat f(W)$ (for the complex structure inherited from $X$ and the handle structure inherited from $\mathcal{H}$). Each map in the isotopy sends each subhandlebody $\mathcal{H}'$ of $\mathcal{H}$ into $f(\mathcal{H}')$. The isotopy is supported in a preassigned neighborhood of $\cl(\mathcal{H}-W)$, and on $W$ it is $C^r$-small (for any preassigned $r\ge 2$).
\end{thm}

\begin{de}\label{gchbody} A {\em generalized Casson handlebody} $\G$ modeled on a 2-handlebody pair $(\mathcal{H},\mathcal{H}^-)$ (or on $\mathcal{H}$ if $\mathcal{H}^-$ is empty) is obtained by replacing 2-handles outside of $\mathcal{H}^-$ with generalized Casson handles using the same attaching maps. If these replacements are all Freedman handles, we will call $\G$ a {\em Freedman handlebody}. The {\em subtowers} of $\G$ are the subsets obtained by cutting back each generalized Casson handle to a finite initial subtower. The {\em handle-compactification} $\bar\G$ of $\G$ is obtained by endpoint compactifying each generalized Casson handle separately (so is not compact if  $\mathcal{H}$ is infinite). We will use analogous terminology for relative handlebodies. Any such $\G$, with a given complex structure, will be called {\em Stein} if it admits a (relative) Stein handle decomposition extending the decomposition on $\mathcal{H}_1\cup\mathcal{H}^-$, such that each subtower of $\mathcal{G}$ inherits a subhandlebody structure. Such a decomposition will be called a {\em Stein handlebody realization} of $\G$ (or of each of its subtowers).
\end{de}

The key step for adding Stein handles by Eliashberg's method is to make the cores {\em totally real}, i.e., lacking complex points. For a compact surface $F$ smoothly immersed in a complex surface $(X,J)$, a point $x \in F$ is a {\em complex point\/} if the tangent space $T_x F$ maps to a complex line in $TX$ (without regard to orientation). Generically, $F$ has only finitely many complex points, all in its interior. Suitably counting these yields a pair of obstructions, which can be reinterpreted as characteristic classes $\rel \partial$. If there are no complex points on $\partial F$, then a nowhere zero tangent vector field $\tau$ to $\partial F$ and outward normal vector field $n$ tangent to $F$ form a complex basis for $TX$ along $\partial F$. If $F$ is oriented, we obtain a characteristic number $c(F)= \langle c_1(X,\tau,n),[F] \rangle$ measuring the first Chern class of $TX|F$ relative to the boundary framing $(\tau,n)$. The Euler number of $TF$ relative to $\tau$ is just the Euler characteristic $\chi(F)$. However, $J\tau$ is never tangent to $F$, so it determines a relative normal Euler number $e(\nu) = \langle e(\nu F, J\tau),[F] \rangle$.

\begin{thm} \label{real} \cite{EH}.
Let $F$ be a compact, connected, oriented surface immersed in a complex surface $X$. Suppose the immersion is generic, i.e., its only singularities are transverse double points in $\inter F$, with no complex points on $\partial F$. Then there is a smooth ambient isotopy of $X$ fixing a neighborhood of $\partial F$, making $F$ totally real, if and only if $c(F)=e(\nu)+\chi (F)=0$. If so, the isotopy can be taken to be $C^0$-small, and supported near a preassigned connected subset of $F$ containing the complex points.
\end{thm}

\noindent Eliashberg and Harlamov prove this for embeddings, by canceling complex points in pairs. The immersed case is no harder, once we rotate the singular points of the immersion to remove any complex points there. For a proof in English see \cite{N}, or ~\cite{Fo1} for a more general version. Note that the invariants $c(F)$ and $e(\nu)+\chi (F)$ are congruent mod 2, since they both reduce mod 2 to the second Stiefel--Whitney number of $TX|F$ relative to the real framing on $TX|\partial F$ induced by $\tau$, $n$ and $J$.

We can now convert abstract 2-handlebodies to Stein generalized Casson handlebodies.

\begin{thm}\label{AnnGCH} Let $\mathcal{H}$ be a 2-handlebody relative to a (possibly empty) $W$. Suppose that $W$ has a proper embedding in a Stein surface with $\partial W$ (strictly) pseudoconvex. The resulting complex structure on $W$ extends to an almost-complex structure on $\mathcal{H}$; let $J$ be any such extension. Then every relative generalized Casson handlebody $\G$ modeled on $\mathcal{H}$ for which each core surface has sufficiently large $g+k_+-k_-$ is a Stein generalized Casson handlebody with respect to some complex structure homotopic to $J$ rel $W$. In particular, every generalized Casson handlebody can be made Stein after suitable refinement. 
\end{thm}

For potential application to explicit constructions, we digress to specify suitable values of $g+k_+-k_-$ when $\mathcal{H}$ is a finite 2-handlebody. In this case, we can exhibit $\mathcal{H}$ by a Legendrian link diagram as in \cite{Ann}, \cite{GS}, although the framing coefficient $f(h)$ of each 2-handle $h$ typically will not be related to the invariants $tb(K)$ and $r(K)$ of its Legendrian attaching circle. Let $c(h,J)$ be the first Chern class of $J$ on $h$, relative to the standard complex tangent trivialization of $(\mathcal{H}_1,J)$ given by the diagram. 

\begin{adden}\label{AnnGCHAdd} For finite $\mathcal{H}$ as above, we can extend the Stein structure over the first stage of a corresponding $\G$ as long as each first stage core satisfies $g+k_+-k_-\ge\frac12 (f(h)+1-tb(K)+|c(h,J)-r(K)|)$. The right-hand side of the inequality is an integer. If the trivializing links of $\G$ are all standard, then the Stein structure extends over all of $\G$ if, in addition, each higher stage has $g+k_+-k_->0$. In particular, these surfaces can be embedded punctured tori, or disks with a single positive double point.
\end{adden}

\begin{proof}[Proof of Theorem~\ref{AnnGCH}]
The extension $J$ exists, and is unique over $\mathcal{H}_1$, since $\mathcal{H}$ has the homotopy of a relative 2-complex and $\SO(4)/{\rm U}(2)\approx S^2$ is 1-connected. Since Eliashberg handle attachment preserves the Stein handlebody condition, it suffices to assume $W=\mathcal{H}_1$ and focus on the 2-handles. We first show that $J$ is homotopic rel $W$ to a complex structure on $\mathcal{H}$. To measure $J$, give $W$ a handle structure so that $\mathcal{H}$ becomes an absolute handlebody, then fix  a complex trivialization of the tangent bundle over the union of 0- and 1-handles. The relative first Chern classes of the 2-handles of $\mathcal{H}$ then determine the homotopy class of $J$. If $\mathcal{H}'$ is a 2-handlebody rel $W$ made from $\mathcal{H}$ by removing its 2-handles (fixing $W$) and replacing them after changing their framings by even numbers of twists, then $\mathcal{H}'$ can be given an almost-complex structure $J'$ with the same relative Chern numbers as $J$ (which can be any integer lifts of the mod 2 class $w_2$). If the framings are sufficiently negative, we can assume $(\mathcal{H}',J')$ is a Stein handlebody rel $W$. We then obtain a handle-preserving immersion $\mathcal{H}\to\mathcal{H}'$ by adding (negative) double points to the 2-handle cores in $\mathcal{H}'$, each double point adding two back to the corresponding normal Euler number (by the displayed formula preceding Definition~\ref{gct} and subsequent text). Pulling back by this immersion gives the required complex structure on $\mathcal{H}$, which we identify with $J$ after homotopy of the latter. (The relative Chern classes are preserved by naturality since the reference trivialization is fixed.)

Now let $F$ be a core surface of the first stage of $\G$, canonically immersed (as preceding Definition~\ref{gch}) in the corresponding 2-handle $h\subset\mathcal{H}$. We wish to make $F$ totally real. After an isotopy of $(F,\partial F)$ in $(h,\partial_-h)$, we can assume the circle $\partial F$ is {\em Legendrian} in $\partial W$, i.e., $\tau$ lies in the contact plane field. Then $J\tau$ also lies in this plane field, so it represents the {\em contact framing} of $\partial F$ in $\partial W$. Let $t(\partial F)$ be the number of twists of this framing relative to the framing canonically induced by $h$. For suitable conventions on the contact solid torus $\partial_-h$, the twisting $t(\partial F)$ is the {\em Thurston--Bennequin invariant} of the Legendrian knot $\partial F$, and the complex framing $(\tau,n)$ is measured by the {\em rotation number} of $\partial F$. (The proof of Addendum~\ref{AnnGCHAdd} below compares these conventions against standard ones.) It is well known that the rotation number can be changed arbitrarily by locally adding zig-zags to a front projection of $\partial F$, so we can set $c(F)=0$ by an isotopy of $(F,\partial F)$ in $(h,\partial_-h)$. (See e.g.\ \cite[Chapter~11]{GS}.) This operation decreases $t(\partial F)$, and we can further decrease it by any even number without changing $c(F)$, using pairs of zig-zags in opposite directions. However, we cannot always increase $t(\partial F)$, which is the ultimate reason for failure of Eliashberg's method to produce Stein structures on arbitrary 2-handlebodies. Since the normal Euler number $e(\nu)$ of $F$ is measured relative to the contact framing $J\tau$, the normal Euler number of $F$ relative to the framing canonically induced by $h$ is $e(\nu)+t(\partial F)$. This is given by $-2(k_+ - k_-)$, by the text below the formula preceding Definition~\ref{gct}. Now the other obstruction from Theorem~\ref{real} is given by $e(\nu)+\chi(F)= -t(\partial F)-2(k_+ -k_-) +1-2g=1-t(\partial F)-2(g+k_+-k_-)$. If $F$ is the core of $h$, the last term vanishes to give the well-known obstruction to making 2-handlebodies Stein: the 2-handle framing should be one less than the contact framing. However, for any sufficiently large $g+k_+-k_-$, the obstruction will be nonpositive and even (since its mod 2 residue agrees with that of $c(F)=0$). Thus, we can set it to zero by decreasing $t(\partial F)$ by an isotopy, then arrange $F$ to be totally real by Theorem~\ref{real}.

To obtain a Stein handlebody realization of the first stage of $\G$, we make each first stage core totally real, then subdivide it as a relative CW-complex. The complex can be thickened to produce a Stein handlebody rel $W$ by Eliashberg's method. For example, \cite[Theorem~8.4]{CE} allows us to successively thicken the handles to be pseudoconvex without disturbing the totally real condition. The Chern classes relative to the fixed reference trivialization are again preserved.
We complete the proof by induction on the number of stages in the subtower under consideration. We inductively assume a given subtower of $\G$ has a Stein handlebody realization. After isotopy, each trivializing link lies in its union of 0- and 1-handles. Applying the previous construction with $W$ replaced by this subtower extends the realization over the next stage. The infinite union $\G$ is then a Stein handlebody of the required form, since it is made by attaching Eliashberg handles to $W$.
\end{proof} 

\begin{proof}[Proof of Addendum~\ref{AnnGCHAdd}]
We adopt the conventions for a standard Legendrian diagram of a Stein domain \cite{Ann}, \cite[Chapter~11]{GS} and follow the previous proof using the standard trivialization on $T\mathcal{H}_1$. Initially, $\partial F$ is the Legendrian attaching circle $K$ of $h$. The twisting $t(\partial F)$ measures the contact framing with respect to the framing induced by the handle $h$, which has coefficient $f(h)$ in the diagram. The Thurston--Bennequin invariant directly gives the coefficient of the contact framing, so $tb(\partial F)=t(\partial F)+f(h)$. Similarly, the rotation number $r(\partial F)$ measures $(\tau,n)$ relative to the standard trivialization, so $c(h,J)=c(F)+r(\partial F)$. But $c(h,J)$ is preserved by our isotopy of $F$ setting $c(F)=0$, which then resets the rotation number $r(\partial F)$ from $r(K)$ to $c(h,J)$, lowering $tb(\partial F)$ by $|c(h,J)-r(K)|$. The first two sentences of the addendum now follow immediately from the formula for the obstruction $e(\nu)+\chi(F)$ in the above proof. (Recall that the obstruction is now even since $c(F)=0$.) To build the higher stages, we can assume the previous stages were built using the standard diagrams \cite{Ann} for increasing $g$, $k_+$ or $k_-$ by one, Figure~\ref{models}(a,b,c) respectively, where the 0-framed curves represent the corresponding trivializing links. Since the latter have $tb=r=0$, we can extend the Stein tower as long as subsequent surfaces have $g+k_+-k_->0$. (We are free to set $c(h,J)=1$ since higher stages do not introduce 2-homology or affect the homotopy class of $J$.) For a different perspective on the most important cases, extending at a torus summand or positive double point, note that a pseudoconvex product neighborhood of a totally real torus is bounded by the standard contact 3-torus, and the standard trivializing link is a pair of framed circles determined by the product structure. These can be taken to be Legendrian with the contact framing, so we can extend as required. If we instead cap one circle by a Stein 2-handle with framing $-1$, we obtain a model for a positive double point and its trivializing circle, and we can again extend as required.
\end{proof}

\begin{figure}
\labellist
\small\hair 2pt
\pinlabel $0$ at 61 16
\pinlabel $0$ at 60 120
\pinlabel $0$ at 257 47
\pinlabel $0$ at 257 120
\pinlabel ${\rm a)}$ at 0 119
\pinlabel ${\rm b)}$ at 212 119
\pinlabel ${\rm c)}$ at 212 47
\endlabellist
\centering
\includegraphics{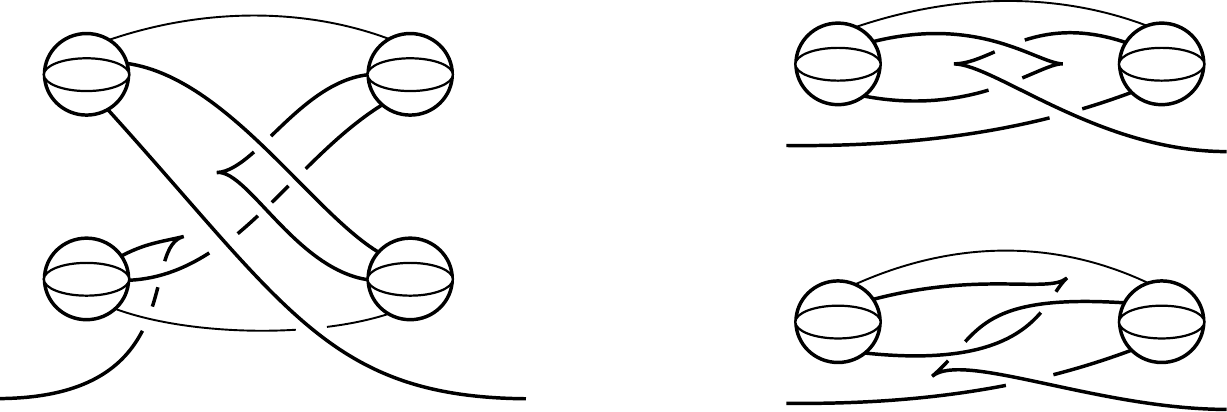}
\caption{Stein standard models for trivializing links.}
\label{models}
\end{figure}
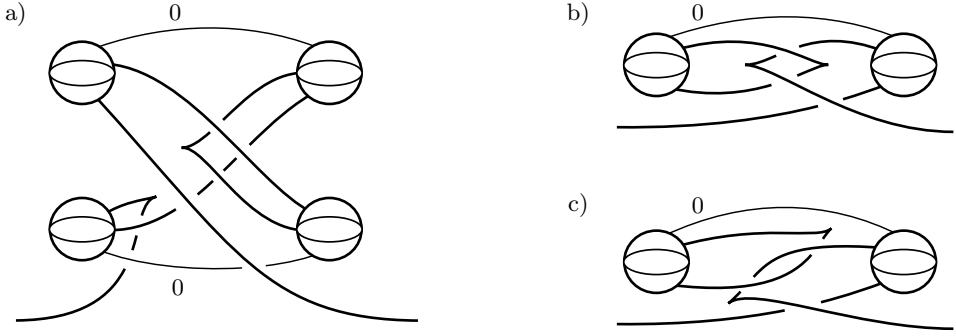

The main theorem of this section generalizes Theorem~\ref{iso1}, ambiently isotoping embedded 2-handlebodies to have Stein interiors.

\begin{thm} \label{iso} Let $\mathcal{H}$ be a 2-handlebody relative to a (possibly empty) $W$, with $W$ properly embedding in a Stein surface so that $\partial W$ is (strictly) pseudoconvex. Given a topological embedding $\mathcal{H}\hookrightarrow X$ into a complex surface, with proper, holomorphic restriction to $W$, suppose $\mathcal{H}$ is collared in $X$ and $\cl\mathcal{H}-\mathcal{H}$ is totally disconnected. Then $\mathcal{H}$ is topologically ambiently isotopic to the handle-compactification $\bar\G$ of a Stein, relative, standard Freedman handlebody modeled on $(\mathcal{H},W)$ (with the embedding respecting the handle structure of $\mathcal{H}$ and complex structure of $X$). The same is true without the Stein conclusion, for $W\subset X$ any smooth, proper 4-manifold pair. In each case, the isotopy sends each subhandlebody of $\mathcal{H}$ into itself, is supported in a preassigned neighborhood of $\cl(\mathcal{H}-W)$, and is $C^r$-small on $W$ (for any preassigned $r$).
\end{thm}

The nonholomorphic case was essentially known to Freedman and Quinn in the 1980s.

\begin{Remarks}\label{isoRem} (a) The condition that $\cl\mathcal{H}-\mathcal{H}$ is totally disconnected can be replaced by the weaker condition that for every $\epsilon >0$, the handles of diameter $\ge \epsilon$ comprise a closed subset of $X$. (This is independent of choice of metric.) As a simple example, start with a 2-handlebody structure on $D^2\times\R^2\subset\R^2\times\R^2=\R^4$ and radially shrink $\R^4$ onto $\inter D^4$. Then $\cl\mathcal{H}-\mathcal{H}$ is a circle, but the proof still applies.

\item{(b)} The set of Freedman handlebodies $\G$ that can appear in the theorem for a fixed $\mathcal{H}\subset X$ is determined as in Theorem~\ref{embFH} on each 2-handle, subject to lower bounds on $g+k_+-k_-$ in the complex case, as in Addendum~\ref{AnnGCHAdd} if applicable. In particular, if the embedding is smooth on some finite (nonrelative) subhandlebody $\mathcal{H}^-$, the corresponding Freedman subhandlebody of $\G$ can be chosen arbitrarily within the constraint of  Addendum~\ref{AnnGCHAdd}. If the isotopy is of $\inter\mathcal{H}$ and not required to be ambient, then the condition on $\cl\mathcal{H}-\mathcal{H}$ is unnecessary, and we can extend to generalized Casson handlebodies as in Remark~\ref{embFHRem} (with the corresponding constraints on $g+k_+-k_-$).
\end{Remarks}

\begin{proof}[Proof of Theorem~\ref{iso}]
Since $\mathcal{H}$ is collared, it can be enlarged within any preassigned neighborhood in $X$ to a slightly larger handlebody $\mathcal{H}'$ relative to $W$ with $\cl\mathcal{H}'-\mathcal{H}'=\cl\mathcal{H}-\mathcal{H}$. We will construct our isotopy in $\mathcal{H}$ so that it extends to a topological ambient isotopy of $\mathcal{H}'$ fixing $\partial \mathcal{H}'$ pointwise. Note that $\cl\mathcal{H}-\mathcal{H}$ is a closed subset of $X$. (Every neighborhood of a point $x$ in $\cl\mathcal{H}-\mathcal{H}$ must intersect infinitely many handles of $\mathcal{H}$, since any finite subhandlebody is closed in $X$. The same then applies to any $x\in\cl(\cl\mathcal{H}-\mathcal{H})$, which then cannot be in $\mathcal{H}$ but must be in $\cl\mathcal{H}$.) It follows that the complement $X^0$ of $\cl\mathcal{H}-\mathcal{H}$ in $X$ is a manifold in which $\partial\mathcal{H}'$ is a closed subset, so the isotopy on $\mathcal{H}'$ extends by the identity over the rest of $X^0$. Since $\cl\mathcal{H}-\mathcal{H}$ is totally disconnected, adding these points back to $X^0$ is locally the same as endpoint compactifying, so the isotopy extends over the rest of $X$ by the identity on $\cl\mathcal{H}-\mathcal{H}$. (For the variation in Remark~\ref{isoRem}(a), prove continuity at $x\in\cl\mathcal{H}-\mathcal{H}$ by finding, for a given $\epsilon>0$, an $\frac{\epsilon}{2}$-small neighborhood $U$ of $x$ in $X$ such that each $y\in U\cap\mathcal{H}'$ lies in a subhandlebody $\mathcal{H}_y$ of $\mathcal{H}'$ with $\mathcal{H}_y-W$ $\frac{\epsilon}{2}$-small, then applying the last sentence of the theorem.)

To construct the required isotopy in $\mathcal{H}'$, note that by Quinn's Handle Straightening Theorem \cite[2.2.2]{Q}, we can assume that the embedding is smooth near the center of each 0-handle of $\mathcal{H}$, so that after shrinking the 0-handles by a topological ambient isotopy, we can assume they are smooth. The same reasoning smooths the 1-handles. (Quinn's theorem gives a smooth neighborhood of the core union the attaching region, so we can presume the shrink preserves each attaching map into $\partial W$.) For each 2-handle, we apply Theorem~\ref{embFH}, completing the proof in the case without complex structures (with the isotopy fixing $W$ pointwise). In the complex case, we refine so that by Theorem~\ref{AnnGCH}, the new Freedman handlebody $\G$ abstractly admits a Stein structure homotopic to the complex structure inherited from $X$. Since $\G$ is properly embedded in the manifold $X-(\cl\G-\G)$, Theorem~\ref{main} completes the proof.
\end{proof}

To construct Stein onions we will need a Stein version of Addendum~\ref{embFHAdd}. Recall that part (a) of that addendum arranged the high stages of the Freedman handle constructed in Theorem~\ref{embFH} to lie in small balls. Part (b) located the entire Freedman handle inside a small topological ball $B$, except for a tubular neighborhood of an annulus in the first stage core, disjoint from $\inter B$. This was done preserving a preassigned subtower $T_n$ of the Freedman handle up to smooth isotopy.

\begin{adden}\label{isoAdd} (a) We can assume the Freedman handles of $\G$ are controlled by small balls as in Addendum~\ref{embFHAdd}(a).

\item[(b)] Suppose that the 2-handles of $\mathcal{H}$ are exhibited as finite towers with 2-handles attached, with totally real first-stage cores $F_i$, and that the resulting  subtower $\mathcal{T}\subset\mathcal{H}\subset X$ is smooth and Stein in the structures inherited from $X$. Then the isotopy in Theorem~\ref{iso} can be assumed to be smooth on $\mathcal{T}$, sending it onto a subtower of $\G$, so that each Freedman handle of $\G$ satisfies the conclusions of Addendum~\ref{embFHAdd} with its first stage core totally real and agreeing near $\partial\mathcal{H}_1$ with the corresponding $F_i$, and the complement in $\G$ of the topological ball interiors of Addendum~\ref{embFHAdd}(b) having pseudoconvex boundary.
\end{adden}

\begin{proof}
Only (b) requires comment. When we apply Theorem~\ref{embFH} in the above proof, we apply it not to each original 2-handle, but to its top stage topological 2-handles. This gives a Freedman handle containing the original subtower, to which Addendum~\ref{embFHAdd} applies. Since each $F_i$ is changed by a smooth isotopy fixing a neighborhood of $\partial F_i$, it can again be assumed totally real because the obstructions continue to vanish. We can add a tubular neighborhood of each resulting totally real annulus to $\mathcal{H}_1$ without disturbing pseudoconvexity (for example, by exhibiting each annulus as a relative 2-complex and using Eliashberg's method). To complete the proof, apply Theorems~\ref{AnnGCH} and \ref{main} as before to the generalized Casson handles remaining after the tubular neighborhoods of annuli are transferred to $\mathcal{H}_1$. (Theorem~\ref{AnnGCH} does not refine $\mathcal{T}$ since it was already given to be a Stein tower with the relevant almost-complex structure.)
\end{proof}

The Stein surfaces $\inter\G$ arising in Theorem~\ref{iso} can be chosen flexibly, as Remark~\ref{isoRem}(b) indicates. The proof also constructs them universally, so that the same diffeomorphism type arises for all triples $W\subset\mathcal{H}\subset X$ that agree up to diffeomorphism on neighborhoods of $\cl\mathcal{H}$ with the almost-complex structures agreeing up to homotopy rel $W$. As an aside, we show that when $W$ is empty, we can relax control of the almost-complex structure: If $\rm H^2(\mathcal{H})$ is finitely generated, we can obtain this universality when the Chern classes $c_1(X)|\mathcal{H}$ are restricted to a bounded region. More generally, fix a cochain $z\in \rm H^2(\mathcal{H},\mathcal{H}_1;\zed)$ and consider all smooth embeddings of the form $f\co V \to Y$, where $Y$ is a complex surface, $V$ is a neighborhood of $\cl\mathcal{H}$ in $X$ and $f(\cl \mathcal{H})$ is closed in $Y$, such that $c_1(X)|\mathcal{H}$ and $f^* c_1(Y)|\mathcal{H}$ are represented by cocycles differing on each 2-handle $h$ by at most $|\langle z,h\rangle|$.

\begin{cor} \label{diff}
For $\mathcal{H}\subset X$ as in Theorem~\ref{iso} with $W=\emptyset$, fix $z\in \rm H^2(\mathcal{H},\mathcal{H}_1;\zed)$ and consider all embeddings $f$ as above. Then the Stein Freedman handlebody $\G \subset X$ obtained from $\mathcal{H} \subset X$ by Theorem~\ref{iso} can be chosen so that each restricted map $f|\G$ is smoothly isotopic (topologically ambiently in $Y$, sending each subhandlebody of $f(\mathcal{H})$ into itself and supported near $\cl(\mathcal{H})$) to a smooth embedding $\hat{f}$, with image $\hat{f}(\G)$ a Stein Freedman handlebody arising as in Theorem~\ref{iso} applied to $f(\mathcal{H}) \subset Y$. This $\G$ can be chosen simultaneously for a finite family of topological embeddings $\mathcal{H}\subset X_k$.
\end{cor}

\begin{adden}\label{diffAdd}
Fix a basis $\{ \gamma_i\}$ for $\rm H_1(\mathcal{H}_1;\zed)$. For each 2-handle $h$, let $m_h=\frac12(|\langle z,h\rangle|-1+\sum |a_i|)$, where $\partial_*h=\sum a_i\gamma_i$. Then $\G$ can be chosen as in Remark~\ref{isoRem}(b), provided that in Addendum~\ref{AnnGCHAdd}, the lower bound for $g+k_+-k_-$ of each first stage core is increased by $m_h$ for the corresponding 2-handle $h$.
\end{adden}

\noindent The addendum is most powerful when $\mathcal{H}$ is smoothly embedded, but can also be applied in the topological case if we have information about the combinatorics of one $\G\subset\mathcal{H}$.

\begin{proof}[Proof of Corollary~\ref{diff} and Addendum~\ref{diffAdd}]
Given $\mathcal{H}\subset X$ and $f\co V\to Y$ as in the corollary, apply Theorem~\ref{iso} to $\mathcal{H}$, working inside $V$. The resulting images $\G\subset \mathcal{H}$ and $f(\G)\subset f(\mathcal{H})$ are independent of choice of $V$, by construction. We will show that when $\G$ is suitably chosen, independently of $f$, Theorem~\ref{AnnGCH} can be applied to $f(\G)$ for the given almost-complex structure $J_Y$ on $Y$. Theorem~\ref{main} then gives the required $\hat{f}(\G)$ as in the proof of Theorem~\ref{iso}. For a finite collection of topological embeddings $\mathcal{H}\subset X_k$, we build the Freedman handlebodies simultaneously, refining as necessary to ensure that they have the same combinatorics, then proceed as before.

Let $\tau$ and $\tau'$ be complex trivializations over $\mathcal{H}_1$ of $J$ and $f^*J_Y$, respectively, for which the relative Chern classes differ on each 2-handle $h$ by at most $|\langle z,h \rangle |$. Then $\tau$ and $\tau'$ need not be homotopic as real trivializations because they may determine different spin structures on $\mathcal{H}_1$. The difference class for these spin structures is an element of $\rm H^1(\mathcal{H}_1;\zed_2)$. Lift it to an integral class whose value on each basis element $\gamma_i$ is either 0 or 1, and use this integral difference class to change $\tau'$ to a new complex trivialization $\tau''$ for $f^*J_Y$. The new relative Chern class $c_1(f^*J_Y,\tau'')$ differs from $c_1(J,\tau)$ on a given $h$ by an integer $\delta_h$ with $|\delta_h|\le|\langle z,h \rangle |+ \sum |a_i|=2m_h+1$. By construction, $\tau''$ and $\tau$ determine the same spin structure on $\mathcal{H}_1$, so they are homotopic as real trivializations. Thus, after homotopy of $J_Y$, we can assume that $\tau''=\tau$ and $f^*J_Y$ agrees with $J$ over $\mathcal{H}_1$. Now each first stage surface $F$ of $\G$ has Legendrian boundary and $c(f(F))-c(F)=\delta_h$ bounded independently of $f$. If each such $F$ has sufficiently large $g+k_+-k_-$, then Theorem~\ref{AnnGCH} applies to each $f(\G)$ as required. More precisely (to prove the addendum), note that $c_1(f^*J_Y,\tau)$ and $c_1(J,\tau)$ over $\mathcal{H}$ both reduce mod 2 to $w_2(TX,\tau)|\mathcal{H}$, so each $\delta_h$ is even. Both sides of the inequality of Addendum~\ref{AnnGCHAdd} are integers, and $f$ increases the right side by at most $\frac12 |\delta_h|$. This is at most $m_h$ if the latter is an integer. Otherwise, increasing the bound by $m_h$ automatically increases it by the integer $m_h+\frac12$. Either way, Addendum~\ref{diffAdd} holds, with the higher stages addressed as for Addendum~\ref{AnnGCHAdd}.
\end{proof}


\section{Stein onions} \label{Onionproof}

In this section we construct Stein onions, i.e., topological mapping cylinders adapted to Stein theory (Definition~\ref{oniondef}). We apply Freedman's method that constructs a homeomorphism from a 2-handle to a Freedman handle. Since we already have Freedman's theorems available, we can significantly shorten the argument in several places. In passing, we also obtain a proof that every open generalized Casson handle is homeomorphic to an open 2-handle (Corollary~\ref{homeo}). 

Our starting point is a 2-handlebody $\mathcal{H}$ with a smooth mapping cylinder structure $\psi$ as in the introduction. This is constructed by induction on the indices of the handles, resulting in a core 2-complex $K$ that intersects each handle radially and respects the product structure on the 1-handles. In place of $W$ from our previous theorems, we work relative to a Stein structure on a subhandlebody $\mathcal{H}^-$ of $\mathcal{H}$ that contains all 0- and 1-handles of $\mathcal{H}$. Then $\psi$ is made from a mapping cylinder structure $\psi^-$ on $\mathcal{H}^-$ by modification in the region $R$ descending radially (with respect to $\psi^-$) from the 2-handles outside $\mathcal{H}^-$. We assume the Stein structure on $\mathcal{H}^-$ has the following compatibility with $\psi$, which can always be arranged if $\mathcal{H}^-$ is built by Eliashberg's method, via an isotopy that is smooth except at the 0-cells: The open 2-cells of $K$ intersect $\mathcal{H}^-$ in totally real surfaces, and $\psi^-$ has a sequence of sublevel sets $\mathcal{H}^-_{\sigma_i}$ with $\sigma_i\to0$ exhibited as Stein handlebodies, with the 2-cells of $K$ intersecting each boundary in a Legendrian link. We wish to study a Stein generalized Casson handlebody $\G$ modeled on $(\mathcal{H},\mathcal{H}^-)$. We assume its first stage consists of Stein tubular neighborhoods of totally real immersed surfaces extending the annuli of $K\cap \mathcal{H}^-$, as follows for $\G$ constructed by Theorem~\ref{AnnGCH} or \ref{iso}. (This is true by construction for Theorem~\ref{AnnGCH}, which also exhibits the first stage subtower as a Stein handlebody. It then follows for Theorem~\ref{iso} since the map $\hat{f}$ of Theorem~\ref{main} used in the last sentence of the proof of Theorem~\ref{iso} is a contactomorphism on subhandlebody boundaries \cite{steindiff}, preserving the Stein handle structure coming from Theorem~\ref{AnnGCH}.) We focus on the case when $\mathcal{H}$ is a finite 2-handlebody. The general case (which we only need for the infinite cases of Theorem~\ref{onion1} and a few applications) is given at the end of the section.

\begin{thm} \label{onion}
Let $\G$ be a Stein generalized Casson handlebody modeled on a finite handlebody pair $(\mathcal{H},\mathcal{H}^-)$ as above. Let  $P\subset \inter\G$ be a closed subset of $\bar\G$. Then there is a Cantor set $\Sigma \subset [0,1]$ containing $\{0,1\}$, and a quotient map $g\co\mathcal{H} \to \bar{\G}$ that is a homeomorphism except for collapsing on the boundaries of 2-handles whose images are not Freedman handles, such that $g$ satisfies the following conditions:
\begin{itemize}
\item[a)] For each $\sigma \in\Sigma-\{0,1\}$, the image $\bar\G_\sigma=g(\mathcal{H}_\sigma)$ is a handle-compactified standard Stein Freedman handlebody containing the image of $P$ under some smooth ambient isotopy of $\G$.

\item[b)] The map $g$ is the identity on $\mathcal{H}^--R$ and on $K\cap\mathcal{H}^-$, and restricts diffeomorphically, with totally real image, to the complement of a point on each other open 2-cell of $K$.
\end{itemize}
\end{thm}

\begin{Remark}\label{sig0}
For infinite $\mathcal{H}$ the same proof applies, up to smooth isotopy of $g|\mathcal{H}^-$ (fixing $K$ setwise) if $\Sigma$ is replaced by a suitable countable subset $\Sigma_0\subset\Sigma$ with a cluster point at 0. To get all of $\Sigma$, see Theorem~\ref{onion2}.
\end{Remark}
 
\begin{proof}[Proof of Theorem~\ref{onion1} for finite $\mathcal{H}$] We wish to turn a given embedding into a Stein onion on a given mapping cylinder $\psi$. By Theorem~\ref{iso} with $W$ empty, the embedding is topologically ambiently isotopic to an embedding $f\co\mathcal{H}\to X$ with image a handle-compactified standard Stein Freedman handlebody $\bar\G$. Theorem~\ref{onion} gives a homeomorphism $g\co\mathcal{H}\to\bar\G$ that agrees with $f$ on $\mathcal{H}^--R$, where $\mathcal{H}^-$ is the union of 0- and 1-handles of $\mathcal{H}$. After ambient isotopy of $f$ rel $\mathcal{H}^--R$, we can assume it agrees with $g$ on the remaining boundary solid tori of $\mathcal{H}$, and then on the remaining 4-balls by the Alexander trick. The mapping cylinder structure provided by Theorem~\ref{onion} then agrees with $\psi$. Since $\Sigma$ is unique up to topological isotopy of $[0,1]$, we can assume it is standard as required by Theorem~\ref{onion1}.
\end{proof}
 
\begin{proof}[Proof of Theorem~\ref{onion}]
We will construct the required map $g\co\mathcal{H}\to\bar\G$ following Freedman's original method, with a modification that guarantees the required control near $K$. Recall that the standard Cantor set $\Sigma_{\rm std}$ consists of those real numbers in $[0,1]$ admitting ternary expansions composed only of the digits 0 and 2, so for example, $1/3=.1=.0\bar 2\in\Sigma_{\rm std}$, where $\bar 2$ denotes an infinite string of twos. (We sometimes write the digit 1 for convenience in describing strings of 0 and 2.) For simplicity of notation, we will parametrize the Cantor set $\Sigma\subset [0,1]$ that we construct as though it were standard. This parametrization will extend to a self-homeomorphism of $[0,1]$, but typically not a self-diffeomorphism. For Remark~\ref{sig0}, set $\Sigma_0=\{0\}\cup\{1/3^n|n\in\zed^{\ge0}\}\subset\Sigma$. We will construct $g$ in pieces, starting from the identity on $\mathcal{H}^--R$. To extend over each remaining 2-handle of $\mathcal{H}$, we will use a partial parametrization as described in  \cite[4.3]{FQ}. That is, we will find an uncountable nest of Freedman handles $G_\sigma$ in each generalized Casson handle $G=G_1$ of $\G$, indexed by $\sigma\in\Sigma -\{0,1\}$. We will extend $g$ over each mapping cylinder level $\partial\mathcal{H}_\sigma$ with $\sigma\in\Sigma -\{0\}$ so that its image intersects $\bar G$ as $\partial_+\bar G_\sigma$. We will also extend $g$ to cover certain collars of boundary regions in each $G_\sigma$. These collars $C_\tau$ will be defined along certain subtowers of $G_\sigma$, which we will index by finite sequences $\tau$ of the digits 0 and 2 obtained by truncating the ternary expansion of $\sigma$. (See Figure~\ref{param}. Note that .2 and .20 will correspond to different subtowers.) We will then have $g$ defined outside a countable collection of ``holes''. Finally, we will use the full power of Freedman's surgery theory to homeomorphically extend $g$ over the holes, completing the construction.

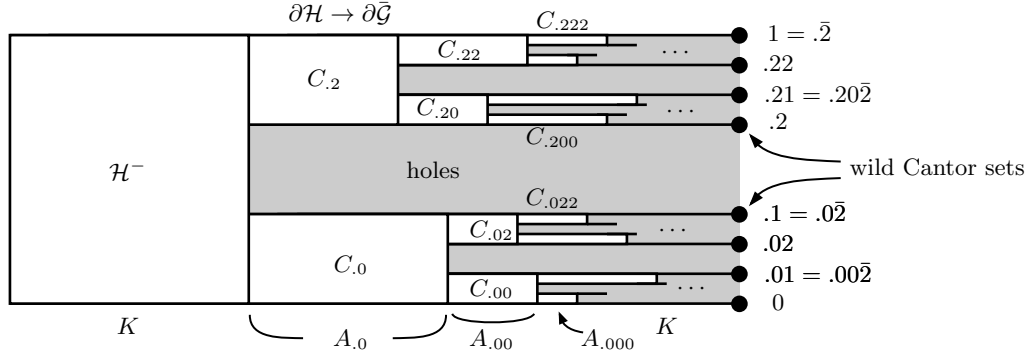
\begin{figure}
\labellist
\small\hair 2pt
\pinlabel $\mathcal{H}^-$ at 45 64
\pinlabel $0$ at 290 14
\pinlabel $.01=.00\bar2$ at 305 25
\pinlabel $.02$ at 290 36
\pinlabel $.1=.0\bar2$ at 300 48
\pinlabel $.2$ at 290 82
\pinlabel $.21=.20\bar2$ at 305 93
\pinlabel $.22$ at 290 104
\pinlabel $1=.\bar2$ at 298 116
\pinlabel $C_{.0}$ at 129 28
\pinlabel $C_{.2}$ at 119 98
\pinlabel $C_{.00}$ at 182 19
\pinlabel $C_{.02}$ at 182 41
\pinlabel $C_{.20}$ at 162 87
\pinlabel $C_{.22}$ at 170 109
\pinlabel $C_{.022}$ at 205 53
\pinlabel $C_{.200}$ at 205 76
\pinlabel $C_{.222}$ at 209 120
\pinlabel $A_{.0}$ at 129 0
\pinlabel $A_{.00}$ at 182 0
\pinlabel $A_{.000}$ at 226 0
\pinlabel $\cdots$ at 258 19
\pinlabel $\cdots$ at 251 41
\pinlabel $\cdots$ at 254 86
\pinlabel $\cdots$ at 252 108
\pinlabel $.01=.00\bar2$ at 305 25
\pinlabel $.02$ at 290 36
\pinlabel $.1=.0\bar2$ at 300 48
\pinlabel wild\ Cantor\ sets at 350 65
\pinlabel $\partial\mathcal{H}\to\partial\bar\G$ at 125 123
\pinlabel $K$ at 45 5
\pinlabel $K$ at 248 5
\pinlabel holes at 160 64
\endlabellist
\centering
\includegraphics{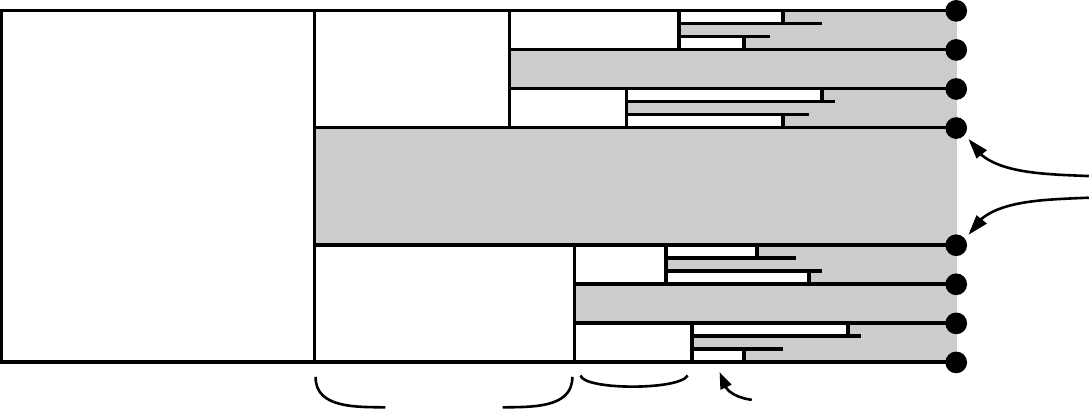}
\caption{Domain of the partial parametrization in $\mathcal{H}$, with labels to indicate the embedding $g$. The vertical axis represents the $I$-coordinate of the mapping cylinder on a given 2-handle, with various levels and collars schematically indicated.}
\label{param}
\end{figure}

\begin{figure}
\labellist
\small\hair 2pt
\pinlabel (a) at 63 101
\pinlabel $\mathcal{H}^-$ at 7 -5
\pinlabel $G$ at 36 -5
\pinlabel $Z$ at 44 28
\pinlabel $\mathcal{T}$ at 106 55
\pinlabel $\mathcal{T}'$ at 82 49
\pinlabel $\mathcal{G}'$ at 115 13
\pinlabel $\mathcal{H}'$ at 132 33
\pinlabel $G^*$ at 131 55
\pinlabel $\mathcal{G}$ at 102 84
\pinlabel (b) at 285 101
\pinlabel $\mathcal{H}^-$ at 233 -5
\pinlabel $\mathcal{H}^-_{\sigma_i}$ at 223 101
\pinlabel $R$ at 240 43
\pinlabel $K$ at 210 53
\pinlabel $A$ at 273 4
\pinlabel $C$ at 277 17
\pinlabel $\mathcal{G}'$ at 323 20
\pinlabel $\mathcal{G}$ at 324 84
\pinlabel $B$ at 303 49
\endlabellist
\centering
\includegraphics{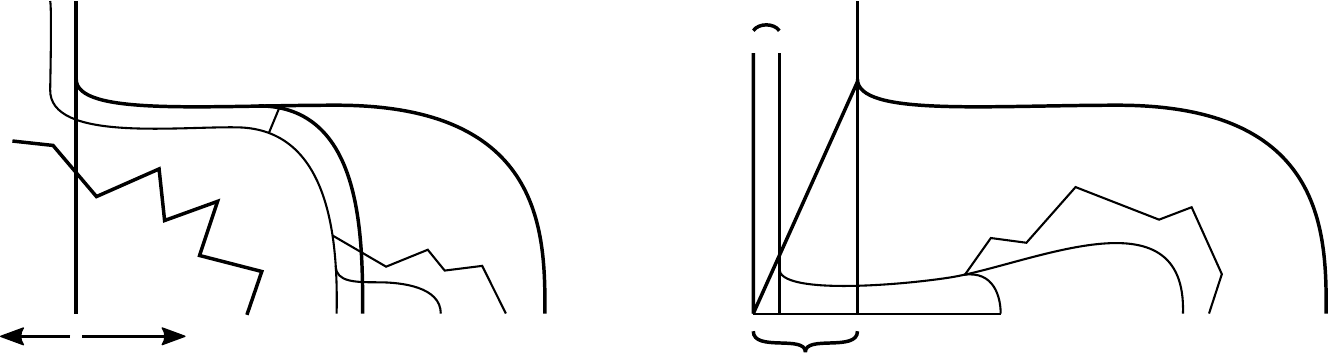}
\caption{Schematic diagrams of the basic (a) and thin (b) embedding procedures. The radial coordinate in $\mathcal{H}^-$ increases toward the right, with $K\cap\mathcal{H}^-$ along the left and bottom edges in (b) (coming from the core of $\mathcal{H}^-$ and the 2-cell of $K$ corresponding to $G$, respectively).}
\label{procedures}
\end{figure}
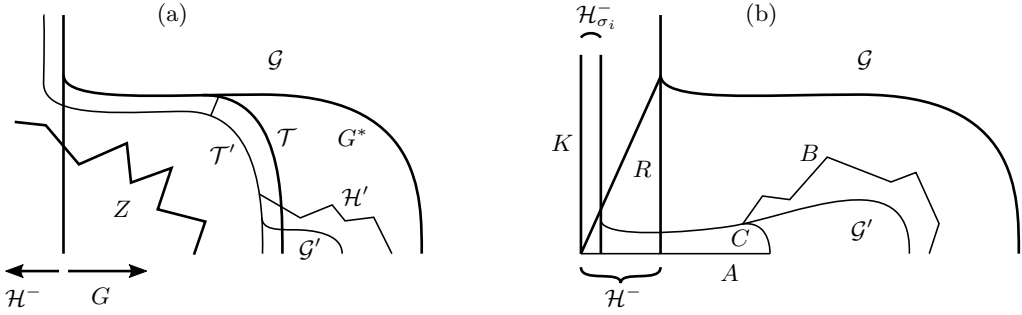

We begin with our basic procedure for constructing a single Stein Freedman handlebody $\G'$ in $\inter\G$, so that it contains a preassigned closed subset $Z$ of $\bar\G$ that lies in $\inter \G$ (Figure~\ref{procedures}(a)). For each generalized Casson handle $G$ of $\G$ attached to $\mathcal{H}^-$, compactness of $Z\cap G$ implies that some finite subtower of $G$ contains $Z\cap G$ disjointly from $\partial_+$. Let $\mathcal{T}$ be the subtower of $\G$ consisting of $\mathcal{H}^-$ together with these towers (of varying height). Recall that $\partial\mathcal{T}$ is pseudoconvex, which is an open condition. Since $\mathcal{H}$ is finite, $\partial\mathcal{T}$ is compact, so some $\psi|([1-\epsilon,1]\times(\partial\mathcal{H}\cap\partial\mathcal{H}^-))$ rescales and extends to a boundary collar $I\times\partial\mathcal{T}\to\mathcal{T}$, disjoint from $Z$, whose levels mapped from $\{t\}\times\partial\mathcal{T}$ are pseudoconvex. (In the infinite case, the same holds without pseudoconvexity, up to isotopy to separate the collar from $Z$.) Let $\mathcal{T}'\subset\inter\mathcal{T}$ be a parallel copy of $\mathcal{T}$ bounded by such a level, a Stein tower with Stein handlebody realization similarly constructed from that of $\mathcal{T}$. Note that $\G$ is made from $\mathcal{T}$ by attaching generalized Casson handles $G^*$ (the higher stages $\cl(\G-\mathcal{T})$ of the original generalized Casson handles). Attach each such $G^*$ to $\mathcal{T}'$ by a vertical collar of $\partial_- G^*$. Lemma~\ref{2h} locates a 2-handle $h_{G^*}$ topologically embedded in each (extended) $G^*$. Let $\mathcal{H}'\subset \G$ be the union of $\mathcal{T}'$ with these topological 2-handles. Apply Theorem~\ref{iso} to $\mathcal{H}'\subset \G$ with $W =\mathcal{T}'$, obtaining the required Stein Freedman handlebody $\G'\supset Z$. By Addendum~\ref{isoAdd}(a), we can assume that for each $n$, the components of $\bar\G'$ minus its first $n$ stages lie in disjoint balls of diameter less than $1/2^n$ whenever they are disjoint from $\mathcal{T}'$. The isotopy of Theorem~\ref{iso} sends $h_{G^*}$ to the corresponding  compactified Freedman handle in $\bar\G'$, so $\pi_1(G^*-\bar\G')\cong \pi_1(G^*-h_{G^*})\cong \zed$ (with the last isomorphism given by Lemma~\ref{2h}).

In addition to this basic procedure, our modification to control $g$ near $K$ requires a ``thin'' variation that changes $Z$ by a smooth ambient isotopy (Figure~\ref{procedures}(b)). Construct $\mathcal{H}'$ as before. After a smooth isotopy rel $\mathcal{H}^--R$ thinning the first stage of $\mathcal{T}'$, we can assume $\mathcal{T}'$ is a Stein tower built from one of the given Stein sublevel sets $\mathcal{H}^-_{\sigma_i}$ ($i$ large). Apply Theorem~\ref{iso} as before, but with $W=\mathcal{H}^-_{\sigma_i}$. By Addendum~\ref{isoAdd}(b), we can assume that  each generalized Casson handle $G$ of $\G-\mathcal{H}^-$ contains an arbitrarily small topological ball $B$ such that $\bar\G'\cap(G-\inter B)$ is a tubular neighborhood $C$ of a totally real annulus $A$ extending an annulus of $K\cap\mathcal{H}^-$, and that the complement of these balls in $\G'$ has pseudoconvex boundary. Since we have only changed $\mathcal{T}'$ by smooth ambient isotopies, it still contains an isotopic copy of $Z$. As before, $\pi_1(G-\bar\G')\cong \pi_1(G-\mathcal{H}')\cong \zed$. (For the last isomorphism, abstractly change each pair $(G^*,h_{G^*})$ to a standard 2-handle pair to transform $(G,G\cap\mathcal{H}')$ to a standard 2-handle pair without changing $\pi_1(G-\mathcal{H}')$.)

\begin{figure}
\labellist
\small\hair 2pt
\pinlabel $\mathcal{H}^-$ at 27 -5
\pinlabel $K$ at -7 122
\pinlabel $C_{.0}$ at 95 31
\pinlabel $C_{.2}$ at 137 154
\pinlabel $C_{.00}$ at 162 32
\pinlabel $C_{.02}$ at 165 65
\pinlabel $C_{.20}$ at 212 121
\pinlabel $C_{.22}$ at 235 148
\pinlabel $C_{.222}$ at 317 156
\pinlabel $A_{.0}$ at 92 10
\pinlabel $A_{.00}$ at 160 10
\pinlabel $A_{.000}$ at 196 10
\pinlabel $\partial\G=\partial\G_1=\partial\G_{.\bar{2}}$ at 190 173
\pinlabel $\partial\G_{.1}=\partial\G_{.0\bar{2}}$ at 214 74
\pinlabel $\partial\G_{.01}$ at 226 13
\pinlabel $\partial\G_{.21}$ at 339 77
\pinlabel $\partial\G_{.221}$ at 309 125
\pinlabel $B_{.00}$ at 240 28
\pinlabel $B_{.0}$ at 290 34
\endlabellist
\centering
\includegraphics{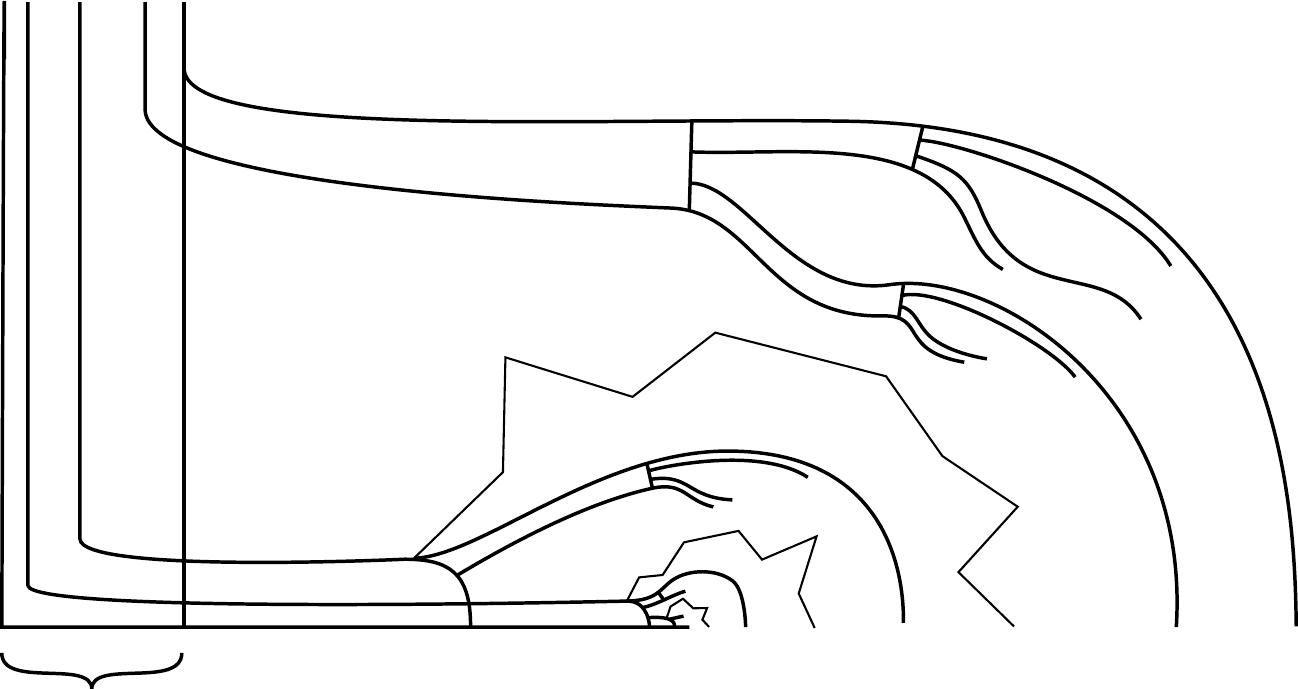}
\caption{Schematic diagram of the image of the partial parametrization. Each $\mathcal{T}_\tau$ with $\tau=.0\dots0$ is depicted as the union of $C_\tau$ with a suitable $\mathcal{H}_{\sigma_i}^-$. For $\tau\ne.0\dots0$ see the next figure.}
\label{image}
\end{figure}
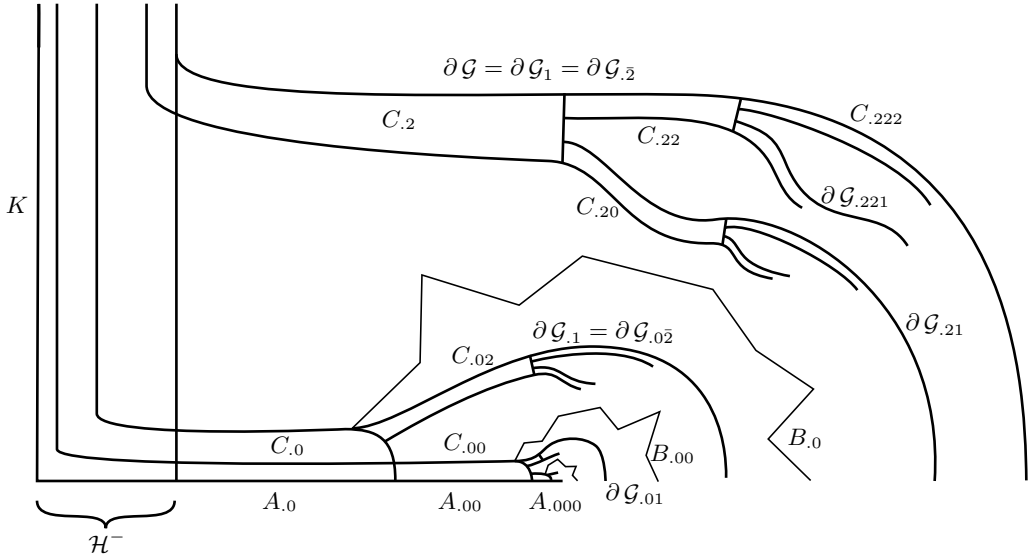

The construction of the partial parametrization proceeds in steps, with the $k^{th}$ step consisting of $2^k$ substeps indexed by ternary expansions $\tau\in\Sigma$ of length $k$. Each substep constructs $g$ on $C_\tau$ and on the entire level indexed by $\tau1=\tau0\bar2$ (where we append digits to the string $\tau$). Each uses the basic procedure above, except for the innermost substeps $\tau=.0\dots 0$, which use the thin variation. In those cases, $C_\tau$ is not a collar but a tubular neighborhood of an annulus $A_\tau$, and first appears one step previously. To set up the construction, which is schematically depicted in Figure~\ref{image}, we define $g|\mathcal{H}^-$ to be the identity (up to subsequent isotopy on $R$). Then we extend $g$ over the boundary of each remaining handle of $\mathcal{H}$ as in Proposition~\ref{fh1}(c) (sending a generalized Bing--Whitehead compactum onto $\bar \G-\G$) so that $\bar \G_1=\bar \G$. The unique substep of the $k=0$ step ($\tau=.$) is innermost, so we apply the thin procedure with $Z=P$, constructing a Stein Freedman handlebody $\G_{.1}$ containing a smoothly isotopic copy of $P$. Then each generalized Casson handle $G$ of $\G-\mathcal{H}^-$ contains a ball $B_{.0}$ and neighborhood $C_{.0}$ of an annulus $A_{.0}$ as above, with the complement $\mathcal{T}_{.0}$ of these balls in $\G_{.1}$ having pseudoconvex boundary. (Note that $\mathcal{T}_{.0}$ is diffeomorphic to $\mathcal{H}^-$, but its structure does not fit the definition of a tower. Each $\mathcal{T}_{.0\dots 0}$ will be obtained from some $\mathcal{H}_{\sigma_i}^-$ by adding thickened annuli.)  By construction, $\partial\G_{.1}$ intersects $\mathcal{H}^--R$ in a level set of $\psi$, which we reindex to be $\partial \mathcal{H}_{.1}$. Isotope $g$ in $R$ and extend it as in Proposition~\ref{fh1}(c) so that $g$ maps $\partial \mathcal{H}_{.1}$ homeomorphically to $\partial \bar\G_{.1}$. Then extend $g$ over $C_{.0}$ as a mapping cylinder. At the $k=1$ step, we construct $\G_{.01}$ and $\G_{.21}$. For the former, we apply the thin procedure inside $\G_{.1}$, to extend a thinned version $W$ of $\mathcal{T}_{.0}$ to $\mathcal{T}_{.00}$, extending the annulus $A_{.0}$ to a new annulus $A_{.00}$ with neighborhood $C_{.00}$, and extending $g$ over $C_{.00}\cup\partial\bar\G_{.01}$ (after isotopy on $R$ rel $\partial \mathcal{H}_{.1}$). For $\G_{.21}$, we apply the basic procedure with $Z=\bar\G_{.1}$, enclosing it in a subtower $\mathcal{T}_{.2}$ of $\G$ that we thin slightly and extend to $\G_{.21}$ (Figure~\ref{tower}). The construction uses a collar of $\partial\mathcal{T}_{.2}$ with pseudoholomorphic levels agreeing with levels of $\psi$ on $\mathcal{H}^-$. Truncate this to a collar of $\partial\G_1\cap\mathcal{T}_{.2}$ as in the figure and let $C_{.2}$ denote the part outside $\mathcal{H}^-$. Extend $g$ over $C_{.2}$ and $\partial\bar\G_{.21}$ by preserving levels. We now have four nested generalized Casson handlebodies $\G_{\tau\bar{2}}$ where $\tau$ has two digits. For $k=2$, we construct another four Stein Freedman handlebodies $\G_{\tau1}$ inside these, using the thin variation when $\tau=.00$ and the basic procedure otherwise. In each latter case, we take $Z$ to be the largest previously constructed $\bar\G_\sigma$ inside $\G_{\tau\bar{2}}$, so that we obtain eight disjoint boundaries. We can then extend our previous regions $C_\tau$ and the map $g$ as before. (This bifurcates each collar, leaving a middle region of truncated levels, as indicated in Figures~\ref{param} and~\ref{image}.) Continuing by induction on $k$, we define $g$ on each level whose index in $\Sigma$ has only finitely many zeroes, and on partial collars of these levels as in Figures~\ref{param} and~\ref{image}. We postpone the proof that $g$ is continuous.

\begin{figure}
\labellist
\small\hair 2pt
\pinlabel $\mathcal{H}^-$ at 3 43
\pinlabel $\bar\G_{.1}$ at 37 9
\pinlabel $\G_1=\G_{.\bar{2}}$ at 110 69
\pinlabel $C_{.2}$ at 44 54
\pinlabel $C_{.22}$ at 111 49
\pinlabel $C_{.20}$ at 75 31
\pinlabel $\mathcal{T}_{.2}$ at 105 32
\pinlabel $\G_{.21}$ at 120 14
\endlabellist
\centering
\includegraphics{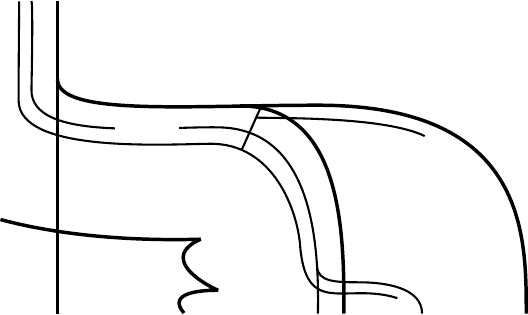}
\caption{Schematic diagram of the tower $\mathcal{T}_{.2}$ thinned and extended to $\G_{.21}$. The other towers with $\tau\ne.0\dots0$ follow the same pattern.}
\label{tower}
\end{figure}
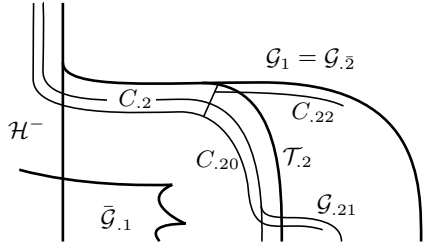

Having defined $g$ on levels corresponding to a dense subset of $\Sigma$, we consider the remaining values $\sigma\in\Sigma$, whose ternary expansions have infinitely many zeroes. (The embedding $\Sigma\subset I$ determining these levels is uniquely determined: We have already specified the lower endpoint $\tau1$ of each middle ``third'', and each upper endpoint $\tau2$ is the limit of a bounded, decreasing sequence of these.) For each such value $\sigma\in\Sigma-\{0\}$, $g$ has already been defined on most of the level. (See Figure~\ref{param}). For each truncation $\tau$ of $\sigma$, the level has been extended as part of the boundary collar of $\mathcal{T}_\tau$, so the image of the level is the boundary of a Stein Freedman handlebody $\G_\sigma$. This is convergent, since the components of the complement of its first $n$ stages  lie in disjoint balls of diameter less than $1/2^n$ once we are beyond the original stages of $\G$. Thus, we can extend $g$ over the entire level, with image $\partial\bar\G_\sigma$. By arranging the towers at each inductive step to have a layer of disks and sufficiently many layers of surfaces (growing exponentially with the length of the ternary expansion of the index), we can guarantee that each such $\G_\sigma$ is a Freedman handlebody (Proposition~\ref{fh1}(a)). Thus, $g$ is a homeomorphic embedding on each level $\sigma\in\Sigma-\{ 0,1\}$ (and for $\sigma=1$ if $\G$ is a Freedman handlebody). Each image is a bicollared topological 3-manifold, by Theorem~\ref{iso} when $\sigma$ has only finitely many zeroes, and by the same method (Proposition~\ref{fh1}(b) as applied at the end of the proof of Theorem~\ref{embFH}, with $T^*_n$ replaced by $\bar\G_{\sigma_i}$ for $\sigma_i\searrow\sigma$) in general. Thus, $\cl G_\sigma=\bar G_\sigma$ is a collared topological 2-handle. At $\sigma=0$, the image of $g$ on its current domain intersects each $G$ of $\G-\mathcal{H}^-$ in the union $A$ of the annuli $A_{.0\dots 0}$. We can assume the balls $B_{.0\dots0}$ arising in the thin procedure intersect in a single point $p\in G$ disjoint from $A$. Then $\cl A=A\cup\{ p\}$ is a topological disk extending to an open 2-cell in $\G$ that is smooth and totally real except at the point $p$ (where it is typically unsmoothable). For each 2-handle $H$ of $\mathcal{H}-\mathcal{H}^-$, we can assume the points at each level $\sigma\in\Sigma-\{0\}$ mapping to $\bar \G_\sigma-\G_\sigma$ (comprising a Bing--Whitehead compactum) become arbitrarily close to the belt circle of the 2-handle $H_\sigma$ as $\sigma$ approaches 0. Then $g$ is defined on the entire core disk of $H$ except for its center point. We send each such point to the corresponding $p$, so that $g$ maps the core homeomorphically to $\cl A$, smoothly except at the center point. Now $g$ is defined on $K$ and has the required properties there. (Freedman's original paper \cite{F} constructs an almost-smooth core disk for a Casson handle that has already been proven homeomorphic to an open 2-handle, by first shrinking his ``central gap" and corresponding hole. In the above description, these come preshrunk. This is necessary since Freedman's shrink would destroy any prearranged pseudoconvexity.)

We complete the definition of $g$ by extending it across the remaining {\em holes}, a countable collection of subsets of $\mathcal{H}$ homeomorphic to $S^1\times D^3$, each exhibited as the product of an interval $[\tau1,\tau2]$ with a solid torus in $\partial \mathcal{H}$. (See Figure~\ref{param}). The image of $g$ currently covers all of $\bar\G$ except for a collection of {\em gaps} in bijective correspondence with the holes. Each gap $\Gamma$ is bounded by the image of the boundary of the corresponding hole, so $\partial \Gamma$ is identified with $S^1\times S^2$ and is bicollared in $\G$ (since the images of the relevant levels of $\mathcal{H}$ are bicollared, respecting the given product structure near the rest of $\partial \Gamma$). In particular, $\Gamma$ is a topological manifold. To understand the gaps more deeply, first set $\tau=.0\dots0$. Each handle of $\G-\mathcal{H}^-$ contains a single gap $\Gamma$ with this $\tau$, whose boundary lies in $\partial\bar\G_{\tau1}\cup\partial C_\tau\cup\partial\bar\G_{\tau2}$ (where $C_.=\mathcal{H}^-$). By construction, $\bar\G_{\tau1}$ lies inside a subtower $\mathcal{T}_{\tau2}$ of $\G_{\tau\bar2}$ that is (slightly thinned) also a subtower of $\G_{\tau2}$ (Figure~\ref{tower}). Thus, $\pi_1(\Gamma)\cong\pi_1(\G_{\tau2}-\bar\G_{\tau1}-C_\tau)\cong\pi_1(\G_{\tau\bar2}-\bar\G_{\tau1}-C_\tau)\cong\Z$. (The last isomorphism is from our initial discussion of the thin procedure.) For $\tau \ne .0\dots0$,  each handle of $\G-\mathcal{H}^-$ contains a gap for each $G^*$ in the basic procedure, but a similar argument shows each has $\pi_1\cong\Z$. In each case, ${\rm H}_i(\Gamma)$ vanishes for $i\ge 2$. (For example, apply Mayer--Vietoris in the contractible ambient generalized Casson handle of $\G$, and note that the $S^2$ factor of $\partial\Gamma\approx S^1\times S^2$ pairs nontrivially with a relative class in the complement of $\Gamma$ given by a Seifert surface for a Whitehead or Bing torus in a level in the current domain of $g$.) Thus, $\Gamma$ is homotopy equivalent to a circle  by \cite[Proposition~11.6C(1)]{FQ}. Now Freedman's work allows us to apply surgery and the s-Cobordism Theorem with fundamental group $\zed$, so a standard argument shows that $\Gamma$ is homeomorphic rel boundary to $S^1\times D^3$. Now we can extend $g$ over each hole. We obtain a well-defined surjection $g\co \mathcal{H}\to\bar\G$  that is injective except for the obvious collapsing from boundaries of 2-handles of $\mathcal{H}$ to the compactifications of generalized Casson handles that are not Freedman handles.

It now suffices to check that $g$ is continuous everywhere, for then it is a closed map by compactness of finite subhandlebodies of $\mathcal{H}$, and hence a quotient map. Each point of $\bar\G$ that is neither the compactification point of some $A$ nor in some $\bar \G_\sigma -\G_\sigma$ with $\sigma\in\Sigma-\{0\}$ has a neighborhood intersecting at most one gap, and so continuity at its preimage point follows immediately.  Each point $p$ in some $\bar \G_\sigma -\G_\sigma$ for which the expansion of $\sigma$ has infinitely many of both zeroes and twos has a system of neighborhoods, each given by a component of the interior of $\G_{\tau\bar2}-\G_\tau-C_\tau$, where $\tau$ is a truncation of $\sigma$. (See Figure~\ref{param}.) Each such neighborhood is the image of a neighborhood in $\mathcal{H}$ (that is a product in the mapping cylinder structure), proving continuity at $g^{-1}(p)$. If $\sigma$ has only finitely many zeroes or twos, then each point $q\in\bar \G_\sigma -\G_\sigma$  lies on the boundary of a gap (or on $\partial\bar{\G}$ if $\sigma=1$). In this case, a similar argument proves continuity at $g^{-1}(q)$ for the restriction of $g$ to the complement of the corresponding open hole, and the extension over the hole is continuous by construction. Finally, a similar argument proves continuity at the preimage of the compactification point of each $A$, so $g$ is continuous everywhere.
\end{proof}

The proof still works if we delete all references to Stein structures and pseudoconvexity. Applying it with $\mathcal{H}^-=S^1\times D^3$ and $\mathcal{H}$ obtained by adding a single canceling 2-handle, we obtain:

\begin{cor} \label{homeo}
Every closed generalized Casson handle has endpoint compactification homeomorphic to a 2-handle modulo a toroidal decomposition in its coattaching region. In particular, every open generalized Casson handle is homeomorphic to an open 2-handle $D^2\times\inter D^2$, by a homeomorphism that is smooth near the boundary and in a neighborhood of $(D^2-\{0\})\times \{0\}$. \qed
\end{cor}

\noindent This corollary was essentially known to Freedman in the 1980s. The above proof (without the Stein structures) is essentially his. Freedman's original proof that a Casson handle is homeomorphic to an open 2-handle first constructed a partial parametrization, then relied on intricate decomposition space arguments, shrinking modified gaps and holes to points to obtain a pair of homeomorphic spaces, and proving that the resulting quotient maps were both approximable by homeomorphisms. Fortunately for us, Freedman's Theorem has already been proved, allowing us to avoid decomposition theory by the suggestion from ~\cite{DF} of seeing directly that the gaps and holes are homeomorphic to each other.

To extend Theorem~\ref{onion} to infinite 2-handlebodies, we introduce another definition.

\begin{de}
Let $\mathcal{H}$ be a relative Stein handlebody. A {\em pseudoconvex collar system} for $\mathcal{H}$ is a proper embedding $\psi_n\co[0,1]\times \partial\mathcal{H}_n\to\mathcal{H}_n$ for each $n$, with $\psi_n(1,x)=x$ for each $x\in\partial\mathcal{H}_n$, and $\psi_m(t,x)=\psi_n(t,x)$ for each $x\in\partial\mathcal{H}_m\cap\partial\mathcal{H}_n$ and $t\in[0,1]$, such that for each $t\in[0,1]$
\item[a)] the level $\psi_n(\{t\}\times\partial \mathcal{H}_n)$ is (strictly) pseudoconvex (in the boundary orientation),
\item[b)] the obvious projection $\psi_n(\{t\}\times\partial \mathcal{H}_n)\to \partial\mathcal{H}_n$ is a contactomorphism, and
\item[c)] removing $\psi_n((t,1]\times\partial\mathcal{H}_n)$ from each $\mathcal{H}_n$ introduces no new boundary intersections.
\end{de}

\noindent The operation in (c) produces thinner versions of the subhandlebodies $\mathcal{H}_n$, with their boundaries required to intersect in the same pattern as in the original handlebody. (For example, a point  $\psi_2(t,x)$ with $x\in\partial\mathcal{H}_2-\mathcal{H}_1$ is not allowed to lie in $\psi_1(\{t\}\times\partial \mathcal{H}_1)$.) This results in a thinner version of $\mathcal{H}$ with a diffeomorphic Stein handlebody structure. (Note that the contact topology of the 2-handle attaching regions, required in Definition~\ref{SHB}(c) of a Stein handlebody, is preserved by (b) above). For $t>0$ the operation in (c), followed by rescaling the interval $[0,t]$, returns a pseudoconvex collar system, as does similarly removing an inner layer from the collars. A finite Stein handlebody always admits a pseudoconvex collar system, since $\bigcup \partial\mathcal{H}_n$ is compact, pseudoconvexity is an open condition, and Gray's Theorem supplies the contactomorphisms. It is not clear to the author whether this remains true for infinite handlebodies (with preassigned Stein structures). However, if a general relative Stein handlebody does admit such a system and we add another layer of handles by Eliashberg's method, \cite[Proposition~10.10]{CE} allows us to extend the system over the new 2-handlebody. (See the beginning of Section~\ref{Iso} above.) Since Theorem~\ref{main} smoothly isotopes 2-handlebodies to be Stein using Eliashberg's method (up to $C^r$-small perturbations), we can assume in that theorem that a given proper collar of $\partial W$ satisfying (a,b) (if one exists) extends to a pseudoconvex collar system of the resulting embedded relative 2-handlebody. The same then applies to the Stein handlebody realization of the Freedman handlebody made by topological isotopy in Theorem~\ref{iso}.

\begin{thm}\label{onion2}
Theorem~\ref{onion} remains true with $\mathcal{H}$ and $\mathcal{H}^-$ allowed to be infinite, up to smooth ambient isotopy (rel $\partial\mathcal{H}$, fixing $K$ setwise and sending each subhandlebody into itself) provided that $\G$ has a Stein handlebody realization with a pseudoconvex collar system for which the image of $\psi_2$ is disjoint from $K$, and some sublevel set $\psi^-(\mathcal{H}^-_\sigma)$ and its intersection with $K$ arise by Eliashberg's method.
\end{thm}

\begin{proof}
Since $\psi^-(\mathcal{H}^-_\sigma)$ is constructed from $K$ by Eliashberg's method, we can thin it by a smooth isotopy (fixing $K$ setwise) to be disjoint from the image of $\psi_2$. It then has its own pseudoconvex collar system suitably disjoint from $K$, allowing iteration. Since the thin procedure for constructing innermost Freedman handlebodies $\G_{.0\dots01}$ ultimately relies on Theorem~\ref{main}, it automatically gives each a pseudoconvex collar system whose $n=2$ collars can be assumed disjoint from each other and from $K$. We can assume the levels of the collars of each $\partial\G_{.0\dots01}$ (including $\G$ itself) agree on $\mathcal{H}^--R$ with those of $\psi$, after an isotopy of the latter rel $K$. We construct the other Freedman handlebodies as in the previous proof. By (c) above, the parallel subtower $\mathcal{T}'$ in the basic procedure of the second paragraph inherits a pseudoconvex collar system from that of $\G$. The resulting Stein Freedman handlebody $\G'$ then inherits a pseudoconvex collar system via Theorem~\ref{iso}. However, since $Z$ need not be compact, arranging it to lie in $\mathcal{T}'$ may require a smooth ambient isotopy of $Z$ in $\mathcal{T}$ (rel $\partial\mathcal{T}$). To preserve disjointness of the levels when we construct each $\bar\G_{\tau1}$, we start from the thinned version $\mathcal{T}'$ of $\mathcal{T}_{\tau}$ with topological handles $h$ attached to form $\mathcal{H}'$ as in the basic construction (Figures~\ref{procedures}(a) and \ref{tower} (the latter for $\tau=.2$)), then jump ahead to find $\mathcal{T}_{\tau2}$ containing this with collar disjoint from  $\mathcal{T}_{\tau}$, and use the collar structure to smoothly push the handles $h$ off of the collar. We then obtain $\bar\G_{\tau1}$ as before, disjoint from the collar. (Every subsequent $\G_\sigma$ with $\sigma>\tau 1$ will be built from this same collar or contain $\mathcal{T}_{\tau2}$.) We take the pseudoconvex collar system of $\bar\G_{\tau1}$ to extend a segment of the previous system as in Figure~\ref{param}, bifurcating the collars as before. Iterating as in the proof of Theorem~\ref{onion} constructs $g$ on all levels in $\sigma$ with only finitely many zeroes, and the rest of the proof applies without change.
\end{proof}

\begin{proof}[Proof of Theorem~\ref{onion1}]
Since we can assume the output of Theorem~\ref{iso} has the required structure, the proof of the finite case (following Remark~\ref{sig0}) still works, where the isotopy of $f$ is ambient by the method of Theorem~\ref{iso}.
\end{proof}

The following addendum to Theorem~\ref{onion1} yields universal diffeomorphism types of Stein onions.

\begin{adden}\label{univ} For the family of embeddings $f\co V\to Y$ in Corollary~\ref{diff}, there are corresponding topological embeddings $\hat{f}$ that are smooth off $K$ and smoothable with support in an arbitrarily small neighborhood of $K$, such that each $\hat{f}$ preserves the levels of the Stein onion structure. In particular, the Freedman handlebodies $\G_\sigma$ and $\hat{f}(\G_\sigma)$ are diffeomorphic for each fixed $\sigma\in\Sigma-\{0\}$.
\end{adden}
 
\begin{proof}
We follow the proof of  Corollary~\ref{diff}, applying Theorem~\ref{iso} to $\mathcal{H}\subset X$. By Addendum~\ref{isoAdd}(b), we can assume each Freedman handle $G$ of the resulting $\G$ lies in a topological ball except for a neighborhood of a totally real annulus $A$, and that the ball lies in a holomorphic ball $B_G$ inside $\mathcal{H}$. After a smooth isotopy, we can assume $f$ is holomorphic on each $B_G$. By the rest of the proof of Corollary~\ref{diff}, we can make the first stage cores $F$ of $f(\G)$ totally real by an isotopy fixing each $B_G$ but modifying the attaching circle of $G$ along zig-zags that extend radially to the 0-skeleton of $K$. We can then arrange $f$ to be a homeomorphism sending $K\cap\mathcal{H}^-$ to a totally real complex. Subdivide each $F$ so that $\partial A$ lies in the 1-skeleton. We can then apply Theorem~\ref{onion} to the resulting subhandlebody in each $B_G$ separately. Since $f|B_G$ is holomorphic, it preserves the resulting Stein onion structure. We can extend these structures over neighborhoods of $K$ and $f(K)$ by Eliashberg's method, successively constructing Stein neighborhoods $\G_{.0\dots01}$ and extending the given pseudoconvex collars. Then $f$ can be isotoped to preserve levels as required.
\end{proof}

\begin{Remark}\label{onionDiff}
This trick avoids difficulties that would arise from trying to directly apply a rel $W$ version of Corollary~\ref{diff} to the proof of Theorem~\ref{onion}. For example, suppose some higher stage component of a Freedman handlebody $\G$ is a smooth 2-handle $h$. After we make the lower stages Stein, we can change the ambient complex structure on $h$ so that $c(h)$ is arbitrarily large, without changing the homotopy class of $J$ on $\G$. We will not be able to replace $h$ by a Stein Freedman handle that is independent of $c(h)$ without modifying the previously constructed Stein structure.
\end{Remark}


\section{Distinguishing smooth structures} \label{Exotic}

We next show that the Freedman handlebody interiors constructed in the previous sections can often be chosen from among a multitude of diffeomorphism types. We do this by adapting \cite{MinGen}, which intensively studied the diffeomorphism types realizable by Stein Casson handlebody interiors modeled on a fixed 2-handlebody. Theorem~\ref{onion} or \ref{onion2} immediately applies to any of these diffeomorphism types, but we can say much more about how the diffeomorphism types vary within the resulting Stein onions. We will start with neighborhoods of surfaces, and then address the general case.

In \cite{MinGen}, constraints on genera of embedded surfaces were studied. In every 4-manifold $X$, every $\alpha\in{\rm H}_2(X)$ is represented by a closed, oriented surface. The {\em genus function} $G\co {\rm H}_2(X)\to\Z^{\ge0}$ assigns to each $\alpha$ the minimal genus of such a surface. It has an analogue $G^{\rm TOP}$ for surfaces topologically embedded with disk bundle neighborhoods, whose values are often strictly smaller than those of $G$. If $X$ is complex, set $G^\C(\alpha)=(\alpha\cdot\alpha+|\langle c_1(X),\alpha\rangle |)/2+1$. The adjunction inequality from gauge theory implies that whenever a surface representing $\alpha$ lies in a Stein open subset $V\subset X$, its genus must be at least $G^\C(\alpha)$, unless it is a sphere that is nullhomotopic in $V$.

\begin{thm}\label{surface}
Let $X$ be a complex surface. Then every $\alpha\in{\rm H}_2(X)$ is represented by a topological embedding of a disk bundle over a surface of any fixed genus $g_0\ge G^{\rm TOP}(\alpha)$, whose image is a Stein onion on the obvious mapping cylinder. When $g_0\ge G(\alpha)$, this can be arranged so that the Stein neighborhoods $\inter\G_\sigma$ for $\sigma\in\Sigma-\{0\}$ realize all values $g\ge\max(g_0,G^\C(\alpha))$ for the minimal genus of the generator of ${\rm H}_2(\inter\G_\sigma)$. Alternatively, we can realize all such values up to a preassigned $N\ge\max(g_0,G^\C(\alpha))$ that is not exceeded. When $G^{\rm TOP}(\alpha)\le g_0< G(\alpha)$, we can still realize arbitrarily large values of $g$, bounded or unbounded.
\end{thm}

\noindent Thus, for fixed $\alpha$ and $g_0$, we realize infinitely many diffeomorphism types of Stein smoothings of the given bundle interior, topologically ambiently isotopic in $X$ but distinguished by the genus function. For each case with $g_0\ge G(\alpha)$, the family of diffeomorphism types can be chosen universally as in Addendum~\ref{univ} for any given bound on $|\langle c_1(X),\alpha\rangle|$. In general, the proof preserves the topological isotopy class of a preassigned bundle, so we could apply this to knotted spheres in $\C^2$, for example. According to \cite[Section~6]{MinGen}, we can interpret the integer $N$ (or infinity in the unbounded case) as the minimal genus of the generator for the almost-smooth core, distinguishing the ``myopic equivalence" classes of these cores for different values of $N$. As in \cite[Corollary~8.5]{MinGen}, counting double points of each sign, rather than genus, would allow us to choose the onion in any of $n$ different ways $\G(1),\dots,\G(n)$ ($n\in\Z^+$ arbitrarily large), such that no $\G_\sigma\subset \G(i)$ smoothly embeds in any other $\G(j)$ (generating its homology).

\begin{proof}

When $g_0\ge G(\alpha)$, we can represent $\alpha$ by a smooth surface $F$ of genus $g_0$. For any $g$ as given, \cite[Example~6.1(b)]{JSG} topologically isotopes and thickens $F$ into a Stein Casson handlebody $\mathcal{C}_g\subset X$ using a single Casson handle with $g-g_0$ first stage double points, all positive, with the generator of ${\rm H}_2({\mathcal C}_g)$ having minimal genus $g$. (The idea is to add double points to $F$ and extend to a Stein Casson handlebody $\mathcal{C}_g$, then smoothly embed $\mathcal{C}_g$ in a different Stein manifold in which $G^\C=g$. The case $g_0=0$ is excluded needlessly in \cite{JSG} since the relevant adjunction inequality still holds; cf.\ \cite[proof of Theorem~3.8]{MinGen}.) We can easily upgrade $\mathcal{C}_g$ to a Stein Freedman handlebody $\G\subset\mathcal{C}_g$ with the same first stage (e.g., by Theorem~\ref{iso} with $W$ the first stage). Then $\G$ also has generator of minimal genus $g$. (It cannot be smaller since $\G\subset\mathcal{C}_g$, but there is a genus-$g$ representative made by smoothing double points.) Now we apply Theorem~\ref{onion} with an upgrade in its proof. Before constructing $\G_{.1}$, refine the cell decomposition of the first stage core surface $F'$ of $\G$ so that it has a 0-cell inside the smooth target ball $B$ in which we will place the topological ball $B_{.0}$. This adds a 0-handle $h\subset\inter B$ and canceling 1-handle to $\mathcal{H}^-$. Add a positive double point to $F'$ and compensate with zig-zags on $\partial h$ so that $F'$ remains totally real outside $h$. Then complete the construction of $\G_{.1}$ as before, relative to $h$. The new first stage is totally real outside $h$, so we obtain the totally real annulus $A_{.0}$ as before (disjoint from $h$). After extending the canceling 0-1 pair into the target ball for $\G_{.01}$, we can move the zig-zags into the latter in preparation for adding another double point and pair of zig-zags. When we construct $\G_{.1}$, we refine as necessary so that it admits  a smooth embedding in $\mathcal{C}_{g+1}$ and hence has minimal genus $g+1$. Continuing in this way, we arrange the innermost Freedman handlebodies $\G_{.0\dots 01}$ to have minimal genus increasing by 1 at each step. If we instead wish the minimal genus to stabilize at some $N$, we set the corresponding first stage equal to $P$ in the theorem so that it has an isotopic copy in each subsequent Freedman handlebody. If $g_0<G(\alpha)$, we only have a topological embedding to begin with, so have less precise control over the minimal genus (due to negative double points), but we can still force it to become arbitrarily large using auxiliary embeddings in ${\mathcal C}_n$, or stabilize it as before (with finite, but uncontrollable, upper bound).
\end{proof}

\begin{thm}\label{surfaceUncountable}
Suppose that $\alpha\cdot\alpha\le0$ in  Theorem~\ref{surface}. Then in each case of that theorem, the onion can be constructed so that each of the required values of minimal genus $g$ is realized by uncountably many diffeomorphism types of Stein neighborhoods $\inter\G_\sigma$ for $\sigma\in\Sigma-\{0\}$.
\end{thm}

 The proof of this theorem relies on a family of exotic $\R^4$ homeomorphs constructed by DeMichelis and Freedman in \cite{DF}. These were realized as Freedman handlebody interiors modeled (after simplification \cite{BG}) on a handle decomposition of the 4-ball with two 1-handles and two 2-handles, one of which becomes a Freedman handle. In particular, they smoothly embed in $\C^2$. It was shown in \cite{Ann} that the Freedman handlebodies are Stein. (Actually, Casson handles were used there, but we can now return to the Freedman handle perspective.) It now immediately follows (\cite{steindiff}, cf.\ Theorem~\ref{main} above) that their handle compactifications embed in $\C^2$ with Stein interior. Let $\r\subset\C^2$ be such a Freedman handlebody with  a corresponding Stein onion structure. DeMichelis and Freedman showed that for some $\sigma_0\in (0,1)$, each diffeomorphism type in the set $\{\r_\sigma\mid \sigma_0\le\sigma\le1\}$ occurs only countably often, implying uncountably many diffeomorphism types (with the cardinality of the continuum in ZFC set theory). Since the Cantor set has the cardinality of the continuum, the same conclusion follows for the Stein open subsets indexed by $\sigma\in\Sigma\cap[\sigma_0,1]$. A deeper analysis applies after end sum with other manifolds (which, in our context, means taking the boundary sum before passing to the interiors of the manifolds): 
 
 \begin{lem}\label{uncountableR4}\cite[Lemma~7.3]{MinGen}.
Let $\{V_\sigma\mid \sigma_0\le\sigma\le1\}$ be a family of open 4-manifolds, such that each compact submanifold of each $V_\sigma$ has a smooth embedding into a closed, simply connected, negative definite 4-manifold. Let $V_\sigma'$ be the end sum of $V_\sigma$ with $\r_\sigma$. Then in the family $\{V'_\sigma\mid \sigma_0\le\sigma\le1\}$, each diffeomorphism type appears only countably often.
 \end{lem}
 
\begin{proof}[Proof of Theorem~\ref{surfaceUncountable}.]
First note that in the above lemma, $V_\sigma$ and $V'_\sigma$ have the same genus function. This is because $\r$ embeds in $\R^4$, giving embeddings $V_\sigma\subset V'_\sigma\subset V_\sigma$ respecting homology in the obvious way. Thus, the Stein Freedman handlebody $\G\subset X$ in the proof of Theorem~\ref{surface} can be connected by a 1-handle to a small copy of $\r$ in $X-\G$ without changing its genus function. The resulting Stein Freedman handlebody $\G'$ is topologically ambiently isotopic to $\G$ (although it is modeled on a different 2-handlebody with more complicated core). Since a minimal genus representative of the generator of homology is compact, the Stein surfaces $\inter\G'_\sigma$ for $\sigma\in \Sigma$ sufficiently close to 1 will all have the same minimal genus of the generator as $\G$, but represent uncountably many diffeomorphism types by the lemma. (Note that $\G$ smoothly embeds in the disk bundle over $F$ with Euler number $e=\alpha\cdot\alpha\le0$, and hence, in the negative definite $|e|$-fold blowup of $S^4$.) To apply this idea to the larger minimal genera at deeper levels, note that each of the lowermost middle thirds $(\tau1,\tau2)$ in $[0,1]$, where $\tau$ is a string of zeroes, is disjoint from the Cantor set $\Sigma$. Thus, it parametrizes levels of $\G$ that we do not require to be pseudoconvex. (See Figure~\ref{param2}.) Embed a Stein copy of $\r$ in each such region and connect it to $\G_{\tau1}$ by an Eliashberg 1-handle so that the collar in $\G$ parametrized by $[\tau02,\tau1]$ fits pseudoconvexly with the $[.2,1]$ collar of $\r$. We can assume that these collars give pseudoconvex levels of the mapping cylinder, after a topological isotopy of the latter. This gives a new Stein onion topologically ambiently isotopic to $\G$, satisfying the theorem.
\end{proof}

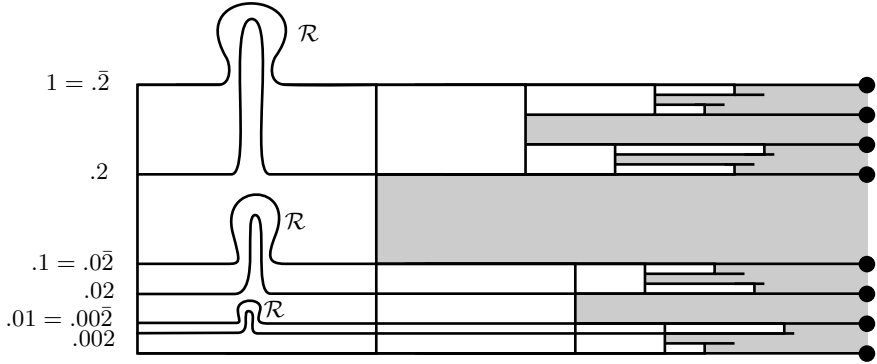
\begin{figure}
\labellist
\small\hair 2pt
\pinlabel $\r$ at 65 124
\pinlabel $\r$ at 60 53
\pinlabel $\r$ at 52 20
\pinlabel $.002$ at -16 9
\pinlabel $.01=.00\bar2$ at -29 17
\pinlabel $.02$ at -14 27
\pinlabel $.1=.0\bar2$ at -24 38
\pinlabel $.2$ at -14 72
\pinlabel $1=.\bar2$ at -22 106
\endlabellist
\centering
\includegraphics{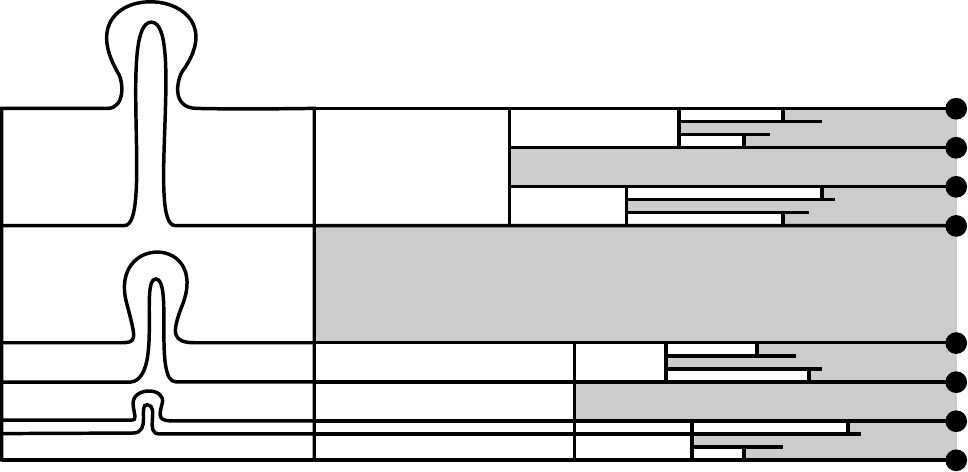}
\caption{Creating uncountably many diffeomorphism types.}
\label{param2}
\end{figure}

The Stein onion created here still comes from the obvious mapping cylinder for the disk bundle, with core surface $F^*\subset X$ smooth and totally real except at one point. However, the Freedman handlebodies $\G'_\sigma$, unlike those of our previous proofs, are modeled on a different handlebody (whose core is made from $F^*$ by connecting it to a contractible 2-complex). If we think of them as modeled on the original handlebody, their boundaries will fail to be smooth on the wild Cantor set $\bar\r_\sigma-\r_\sigma$.

These techniques can be applied whenever we isotope a topologically embedded 2-handlebody to a Stein onion as in Theorem~\ref{onion1}. For example:

\begin{thm}\label{infDiff}
Let $\mathcal{H}\subset X$ be as hypothesized in Theorem~\ref{onion1}.
\item[a)] Suppose $A\subset{\rm H}_2(\mathcal{H})$ is a finitely generated direct summand. Then the resulting Stein onion can be arranged so that $A$ is the subset of ${\rm H}_2(\mathcal{H})$ consisting of all classes $\alpha$ with $G(\alpha)$ for the subset $\G_\sigma$ bounded as $\sigma\to0$.
\item[b)] If every finite subhandlebody of $\mathcal{H}$ smoothly embeds in a closed, simply connected, negative definite 4-manifold, we can assume the Stein surfaces $\inter\G_\sigma$ for $\sigma\in\Sigma-\{0\}$ represent uncountably many diffeomorphism types. If $\mathcal{H}$ is finite, it suffices for it to topologically embed in such a smooth, negative definite manifold.
\end{thm}

\begin{proof}
For (a), slide handles so that $A={\rm H}_2(\mathcal{H}')$ for some subhandlebody $\mathcal{H}'$. Apply Theorem~\ref{iso} and let $P$ be the union of a finite collection of surfaces in the image $\G'$ of $\mathcal{H}'$ generating $A$. Construct the Freedman handlebodies $\G_\sigma$ so that they embed in models taken from \cite[Lemma~3.2]{MinGen} (for $n=2$, $(H_2,H_1)=(\mathcal{H},\mathcal{H}')$ and the trivial cover $\pi=\id_{\mathcal{H}}$), relative to $\mathcal{H}'$ via \cite[Addendum~3.3]{MinGen}. Letting $k_1$ increase without bound as $\sigma\to0$ guarantees that each $\alpha$ outside of $A$ has $G(\alpha)$ increase without bound. For (b), extend $\G$ and apply Lemma~\ref{uncountableR4} as in the proof of Theorem~\ref{surfaceUncountable}. If $\mathcal{H}$ is finite, the auxiliary topological embedding can be refined to a smooth embedding of a suitably fine $\G$.
\end{proof}

Use of \cite[Lemma~3.2]{MinGen} in (a) guarantees that the subgroup $A\subset{\rm H}_2(\inter\G_\sigma)$ is diffeomorphism invariant, characterized as the span of classes $\alpha$ such that $M(\alpha)=\max(G(\alpha),|\alpha\cdot\alpha|)\le k_1$ (for a suitably large $k_1$). If $A$ is a proper subgroup, we obtain infinitely many diffeomorphism types $\inter\G_\sigma$, distinguished by the minimal value of $M(\alpha)$ for $\alpha\in{\rm H}_2(\mathcal{H})-A$. We can actually make any preassigned filtration by finitely generated summands diffeomorphism invariant in this way by more intensive use of \cite[Lemma~3.2]{MinGen}, or even vary the filtration as $\sigma\to0$. If (b) also applies, we can realize each value of the minimum by uncountably many Stein diffeomorphism types of $\G_\sigma$. If we weaken the mapping cylinder structure so that whenever $\sigma<\sigma'$ we have an inclusion $\bar\G_\sigma\subset\bar\G_{\sigma'}$ that may not be into the interior, then (b) gives uncountably many diffeomorphism types with the same genus function as each $\G_{.0\dots01}$. (Apply the proof of Theorem~\ref{surfaceUncountable} to get an uncountable family of Stein Freedman handlebodies whose boundaries all agree outside a neighborhood of $\r$.) All of these variations apply with universal families of diffeomorphism types as in Addendum~\ref{univ}. We can sometimes obtain stronger uncountability results without negative definiteness, for example, Stein onions whose subsets $\inter\G_\sigma$ are pairwise nondiffeomorphic, Theorem~\ref{R4}(b).

As a concrete application, consider domains of holomorphy (Stein open subsets) in $\C^2$.

\begin{cor}\label{C2}
Let $\mathcal{H}$ be a 2-handlebody with a topologically collared embedding in $\C^2$ (or a possibly infinite blowup of $\C^2$), with $\cl\mathcal{H}-\mathcal{H}$  totally disconnected. Then $\mathcal{H}$ is topologically ambiently isotopic to a Stein onion exhibiting uncountably many diffeomorphism types of topologically ambiently isotopic (rel $K$) Stein neighborhoods, forming a neighborhood system of the core 2-complex $K$ if $\mathcal{H}$ is finite. \qed
\end{cor}

\noindent The hypothesis for Theorem~\ref{infDiff}(b) is satisfied here once we refine $\mathcal{H}$ to a smoothly embedded Freedman handlebody. We have considerable control of minimal genera; cf.\ \cite[Corollary~7.5]{MinGen}. Every Stein open subset of blown up $\C^2$  satisfies the above conclusion after (nonambient) isotopy (cf.\ proof of Corollary~\ref{char}).


\section{Pseudoconvex subsets}\label{Psc}

We now examine various topological notions related to pseudoconvexity. We will consider topologically pseudoconvex 3-manifolds and the complementary notion of pseudoconcavity. We will see that these are common phenomena. First, however, we show that every tamely embedded 2-complex $K$ in a complex surface can be topologically isotoped to an embedding that is pseudoconvex in the sense that it is a nested intersection of Stein surfaces. (Tame means that $K$ is the core of some topologically embedded 2-handlebody, ruling out wild phenomena analogous to Alexander's horned sphere.) In particular, a finite 2-complex becomes a Stein compact. This sharpens the main theorem of \cite{JSG}, which showed that $K$ is $C^0$-small topologically isotopic to a 2-complex with a Stein neighborhood homeomorphic to the interior of a handlebody with core $K$. (In that construction, the neighborhood was a Casson handlebody interior, so its set-theoretic boundary was not a manifold.)

\begin{cor} \label{2cplx}
Let $K \subset X$ be a tame 2-complex in a complex surface. Then after subdivision, $K$ is $C^0$-small topologically ambiently isotopic to the core $K'$ of a Stein onion. In particular, $K'$ is smooth and totally real except for one nonsmooth point of each 2-cell (of the subdivision), and is the intersection of a nested family of Stein neighborhoods that are topologically ambiently isotopic to each other $\rel K'$. These realize infinitely many diffeomorphism types if ${\rm H}_2(K)\ne0$ and uncountably many if $X$ is $\C^2$ or a blowup of it.
\end{cor}

\noindent The nonsmooth points are far from being conelike; see text following Theorem~\ref{onion1}.

\begin{proof}
Given $\epsilon$, subdivide $K$ so that each cell has diameter less than $\epsilon/3$. Then after narrowing the original handlebody neighborhood $\mathcal{H}$ of $K$, we can assume it is collared and is a union of subhandlebodies with diameter less than $\epsilon$. If $\cl\mathcal{H}-\mathcal{H}$ is totally disconnected, we can immediately apply Theorems~\ref{iso} and \ref{onion} or  \ref{onion2}. Since the resulting ambient isotopy sends subhandlebodies into themselves, it is $\epsilon$-small. If $\cl\mathcal{H}-\mathcal{H}$ is not totally disconnected, we can avoid the need for it by choosing the initial subdivision so that Remark~\ref{isoRem}(a) applies. The diffeomorphism types are distinguished by Theorem~\ref{infDiff} (with $A=0$) and the paragraph following its proof.
\end{proof}

Recall that a {\em Stein compact} in a complex surface is a compact subset $Q$ that has a Stein neighborhood system. It is a classical fact that such a $Q$ has Stein interior (e.g.\ \cite[Corollary~1.5.2]{HL}). It is easily checked that such a $Q$ is a collared topological 4-manifold if and only if its set-theoretic boundary is a bicollared topological 3-manifold. In this case, we call $Q$ a {\em topologically pseudoconvex} 4-manifold. Theorem~\ref{onion1} guarantees that a topologically embedded, compact 2-handlebody can be made topologically pseudoconvex by topological isotopy. (Shrink it if necessary to create a collar, apply the theorem and then shrink onto some $\bar\G_\sigma$ for which the ternary expansion of $\sigma$ has no ones and infinitely many zeroes.) More strongly, the resulting collar determines an uncountable system of Stein neighborhoods of the resulting $Q$ with topologically pseudoconvex closure. We extend the notion to 3-manifolds:

\begin{de}\label{deftoppsc}
A topological embedding (resp.\ immersion) $f\co M\to X$ of a closed 3-manifold into a complex surface will be called {\em topologically pseudoconvex} if there is a topologically pseudoconvex 4-manifold $Q$ in some complex surface $X'$, and a homeomorphic identification of $M$ with $\partial Q$ so that $f$ extends to a holomorphic embedding (resp.\ immersion) of some neighborhood of $\partial Q$ in $X'$.
\end{de}

Topologically pseudoconvex 3-manifolds share some of the familiar properties of hypersurfaces that are pseudoconvex (and holomorphically fillable) in the usual smooth sense. A smooth (strictly) pseudoconvex hypersurface $N$ is canonically oriented. If the ambient manifold is Stein, $N$ cuts out a compact Stein domain (so in our dimension of interest $N$ is topologically pseudoconvex.) Similar reasoning shows:

\begin{prop}\label{pscbasics}
A topologically pseudoconvex immersion canonically orients $M$. A topologically pseudoconvex embedding into a Stein surface $V$ cuts out a topologically pseudoconvex region.
\end{prop}

\begin{proof}
Since ${\rm H}_3(V)=0$, every embedding of $M$ cuts $V$ into two components. Since $V$ and $\inter Q$ are Stein, each admits a plurisubharmonic function that is proper and bounded below. The maximum of two plurisubharmonic functions has a plurisubharmonic smoothing \cite[Corollary~3.15]{CE}. Thus, after adding a constant to one function, we can graft to obtain a proper plurisubharmonic function on one component $U$ of $V-f(M)$. Then $U$ is Stein, so it can have only one end, and $\cl U=U\cup f(M)$ is compact. Similarly, each sufficiently small Stein neighborhood of $Q$ determines a Stein neighborhood of $\cl U$. A topologically pseudoconvex immersion identifies $M$ with some $\partial Q$, orienting it by the boundary orientation. If some other such identification with $\partial Q'$ gave the opposite orientation, we could graft a pair of plurisubharmonic functions to construct a 2-ended Stein neighborhood of $\partial Q$, with forbidden 3-homology.
\end{proof}

Topologically pseudoconvex 3-manifolds are much more common than their smooth counterparts. We illustrate this now with some examples, then explore more deeply in \cite{TPC}. Since our examples all result from Theorem~\ref{onion1}, they automatically come with topological bicollarings exhibiting uncountably many other such 3-manifolds. These 3-manifolds will often be unsmoothable and smoothly distinct from each other in various ways (e.g., Theorem~\ref{R4}(b) or \cite[first sentence of Corollary~6.5]{MinGen}), although they are topologically parallel.

\begin{thm}\label{toppsc}
Let $M$ be a closed, oriented 3-manifold.
\begin{itemize}
\item[a)] $M$ has a topologically pseudoconvex immersion in $\C^2$.
\item[b)] There is an integer $n$ such that $M$ has a topologically pseudoconvex embedding in every closed, simply connected, complex surface $X$ with $b_\pm(X)> n$.
\item[c)] Suppose a bicollared topological embedding of $M$ in a complex surface $X$ bounds a compact region $C\subset X$. If $C$ is smoothable as an abstract topological 4-manifold  (e.g.\ if the embedding of $M$ is smooth) then there is a smooth embedding $Y\to \inter C$ (smoothed as an open subset of $X$), where $Y$ is a boundary sum of copies of $S^1\times D^3$, and an ambient connected sum $M\#\partial Y$ that is topologically ambiently isotopic in $X$ to a topologically pseudoconvex embedding.
\end{itemize}
\end{thm}

\begin{proof}
We can realize $M$ as the boundary of a 2-handlebody $\mathcal{H}$ consisting of 2-handles attached to a 4-ball with even framings. (See e.g.\ \cite[Theorem~5.7.14]{GS} for further discussion.) For (a), we can easily construct a smooth immersion of $\mathcal{H}$ in $\C^2$. (Control the framings by adding double points and using the displayed formula near the beginning of Section~\ref{GCH}.) Pull back the complex structure to $\mathcal{H}$ and apply Theorem~\ref{onion1} to an isotopic copy of $\mathcal{H}$ in $\inter\mathcal{H}$. For (b), double $\mathcal{H}$ to get an embedding of $\mathcal{H}$ in $Z=\# n S^2\times S^2$ (where $n$ is the number of 2-handles of $\mathcal{H}$). By Freedman's classification of simply connected, topological 4-manifolds \cite{F}, every closed, simply connected 4-manifold $X$ with $b_+>n$ and $b_->n$  has a topologically embedded copy of $Z-\{p\}$, and hence, of $\mathcal{H}$. For $X$ complex, apply Theorem~\ref{onion1}.

For (c), we abstractly smooth $C$ and then present it as a handlebody. The cocores of the 3- and 4-handles form a 1-complex with a regular neighborhood whose closed complement is a 2-handlebody. After sliding 1-cells, we may identify the 1-complex as a wedge of circles attached by an arc to $\partial C=M$. Smooth this 1-complex in $X$ by Quinn's Handle Straightening Theorem \cite[2.2.2]{Q}, and identify its regular neighborhood as $Y$ attached to $M$. Apply Theorem~\ref{onion1} to the complementary 2-handlebody.
\end{proof}

Finally, we consider the corresponding notion of concavity.

\begin{de}\label{toppscv}
A compact 4-manifold $P$ topologically embedded in a complex surface will be called {\em topologically pseudoconcave} if each component of $\partial P$ is topologically pseudoconvex with reversed orientation.  A complex surface $V$ will be called {\em topologically pseudoconcave} if it is biholomorphic to the interior of such a $P$.
\end{de}

\noindent Equivalently, $V$ is made from a closed complex surface $X$ by cutting out finitely many topologically pseudoconvex 4-manifolds. (By definition, each end of $V$ has a topological open collar that is holomorphically identified with an exterior collar of some topologically pseudoconvex 4-manifold, so we can glue to recover $X$.) Part (a) of the following theorem exhibits many pseudoconcave examples homeomorphic to $\R^4$ and $I\times S^3$. These are made via Theorem~\ref{onion1}, so each end is topologically collared by a 3-manifold, with uncountably many topologically pseudoconvex levels $\sigma$, directed {\em inward}. The construction also completes a discussion from the end of the previous section (b).

\begin{thm}\label{R4}
(a) Every closed, simply connected, complex surface $X$ contains an uncountable family of pairwise nondiffeomorphic, topologically pseudoconcave, manifolds homeomorphic to $\R^4$ (or alternatively $I\times S^3$ with nondiffeomorphic interiors), nested with the order type of a Cantor set minus a countable boundary subset.

\item[(b)]  Every simply connected 2-handlebody $\mathcal{H}$ with $b_2(\mathcal{H})>0$ and $\partial\mathcal{H}$ diffeomorphic to $S^3$ underlies some Stein Freedman handlebody $\G$ for which no two levels $\partial\bar\G_\sigma$, $\sigma\in(0,1]$, of any given mapping cylinder can have disjoint diffeomorphic neighborhoods. In particular, the subsets $\inter\G_\sigma$ are pairwise nondiffeomorphic.
\end{thm}

\begin{proof}
We prove (b) by the method of \cite[Theorem~7.7]{MinGen} (which also controls the genus function on $\G_\sigma$). By Freedman's classification, our hypotheses guarantee that we can extend $\mathcal{H}$ to a manifold homeomorphic to some $Z=\#\pm\C P^2$ by attaching a topological manifold with a nonstandard, definite intersection form. Theorem~\ref{iso} isotopes $\mathcal{H}$ to a Stein Freedman handlebody $\bar\G\subset Z$ respecting any given complex structure near $\mathcal{H}$ (for example the obvious structure on $Z-\{p\}$). If any two levels $\partial\bar\G_\sigma$ had disjoint diffeomorphic neighborhoods, we could extend $Z-\bar\G_\sigma$ to a manifold with a periodic end and nonstandard, definite intersection form, contradicting Taubes \cite{Tb}, cf.~\cite[Theorem~9.4.10]{GS}.

If $X$ in (a) is even, it must satisfy the inequality $11|\sigma(X)|\le8b_2(X)$. (This follows from the standard inequalities $2\chi(X)+3\sigma(X)\ge0$, $\chi(X)\ge3\sigma(X)$, $b_-(X)\ge1$, and $b_+(X)\ge3$ unless $\sigma(X)=0$.) By Freedman's classification, $X$ must then be homeomorphic to $\#\pm\C P^2$ or $\pm\# mK3\# nS^2\times S^2$. Thus, it is homeomorphic to a 2-handlebody $\mathcal{H}$ with a 4-handle attached to its $S^3$ boundary. Isotope the topological embedding of $\mathcal{H}$ to a Stein onion $\bar\G\subset X$. We can assume after refining that $\G$ also has an embedding as in the previous paragraph. The subsets $\inter\G_\sigma$ are then pairwise nondiffeomorphic, as are the regions $R_\sigma=X-\bar\G_\sigma$. For most $\sigma\in\Sigma$, the latter are pseudoconcave subsets homeomorphic (but not diffeomorphic) to $\R^4$. For $I\times S^3$, remove a convex open ball from each $\cl R_\sigma$. Such a ball is topologically ambiently isotopic to a standard ball in the topological $\R^4$ structure by 0-handle smoothing \cite{Q}. The resulting interiors are pairwise nondiffeomorphic as before.
\end{proof}

Unlike our Stein exotic smoothings of $\R^4$, the above pseudoconcave manifolds $R_\sigma$ cannot embed smoothly in $\R^4$ \cite[Theorem~9.4.3]{GS}, and they require infinitely many index-3 critical points in any proper Morse function \cite{T}; see also \cite[Exercise~9.4.20(a)]{GS}. In contrast, Kasuya and Zuddas \cite{KZ} exhibited a smoothly pseudoconcave complex structure on $B^4$. This gives a complex structure on the standard $\R^4$ that is pseudoconcave in a related smooth sense. However, it fails our definition since the boundary $S^3$ is overtwisted, so cannot be filled on the outside. In fact, if a smooth pseudoconcave boundary sphere can be filled on the outside, the filling must be a ball (up to blowups) \cite[Theorem~16.5]{CE}. Thus, the pseudoconcave manifold must be a punctured complex surface; in particular, not contractible.

For complex surfaces with arbitrary fundamental group, we can split by the method of Theorem~\ref{toppsc}(c):

\begin{cor}\label{split} Every closed complex surface $X$ has a topological decomposition $P\cup_\partial Q$ with $Q$ topologically pseudoconvex and $P$ homeomorphic to $Y=\natural n S^1\times D^3$. There is a smoothly pseudoconvex embedding $Y\subset P$, splitting $X$ into $Y$ and $Q$ separated by a pseudoconcave topological collar $I\times \partial Y$. \qed
\end{cor}

\end{document}